\documentclass[10pt,reqno]{article}
\usepackage{amsmath,amssymb,amsthm}
\usepackage{esint}
\usepackage{bm}
\usepackage{subfigure}
\usepackage{graphicx,color}

\def\Bu{{\bf u}}
\def\Bv{{\bf v}}
\def\Bw{{\bf w}}

\def\BN{{\bf N}}

\newcommand{\divg}{\nabla \cdot}
\newcommand{\curl}{\nabla \times}

\DeclareMathAlphabet{\itbf}{OML}{cmm}{b}{it}

\def\be{{{\itbf e}}}
\def\bn{{{\itbf n}}}

\newcommand{\Dc}{{\Omega\backslash\overline D}}

\newcommand{\bN}{\mathbf{N}}
\newcommand{\bU}{\mathbf{U}}
\newcommand{\RR}{\mathbb{R}}
\newcommand{\R}{\mathbb{R}}

\newcommand{\II}{\mathbb{I}}
\newcommand{\ID}{\mathbf{I}_2}
\newcommand{\DD}{\mathbb{D}}

\newcommand{\K}{{\kappa}}
\newcommand{\dis}{\displaystyle}
\newcommand{\OL}{\mathcal{L}}

\newcommand{\bu}{\mathbf{u}}
\newcommand{\bv}{\mathbf{v}}
\newcommand{\bw}{\mathbf{w}}

\newcommand{\bz}{\mathbf{z}}
\newcommand{\bx}{\mathbf{x}}
\newcommand{\by}{\mathbf{y}}

\newcommand{\E}{\mathcal{E}}
\newcommand{\CC}{\mathbb{C}}
\newcommand{\MM}{\mathbb{M}}
\newcommand{\bGam}{\mathbf{\Gamma}^\omega}

\newcommand{\ds}{\displaystyle}
\newcommand{\NN}{\mathbf{N}}

\renewcommand{\S}{\mathcal{S}}

\newcommand{\OH}{\mathcal{H}}

\newcommand{\I}{{\mathcal{I}}}
\newcommand{\IWF}{\mathcal{I}_{\rm WF}}

\newcommand{\Kcal}{\mathcal{K}}

\def\nm{\noalign{\medskip}}

\newcommand{\beq}{\begin{equation}}
\newcommand{\eeq}{\end{equation}}

\newcommand{\bg}{\mathbf{g}}

\newcommand{\bnu}{\bm{\nu}}
\newcommand{\bmu}{\gamma}
\newcommand{\EE}{\mathbb{E}}
\newcommand{\signoise}{\sigma_{\rm{noise}}}
\newcommand{\sigmu}{\sigma_\gamma}

\newtheorem{thm}{Theorem}[section]
\newtheorem{cor}[thm]{Corollary}
\newtheorem{lem}[thm]{Lemma}
\newtheorem{prop}[thm]{Proposition}
\newtheorem{defn}[thm]{Definition}
\newtheorem{rem}[thm]{Remark}

\numberwithin{equation}{section}

\newcommand{\pathfigures}{Figures/}
\graphicspath{{\pathfigures}}

\begin{document}
\title{Localization, Stability, and Resolution of Topological Derivative
 Based Imaging Functionals in Elasticity
\thanks{\footnotesize This work was supported by the ERC Advanced Grant Project MULTIMOD--267184 and Korean
 Ministry of Education, Science, and Technology through grant NRF 2010-0017532.}
}

\author{
Habib Ammari
\thanks{\footnotesize Department of Mathematics and Applications, Ecole Normale Sup\'erieure, 45 Rue d'Ulm,
 75005 Paris, France (habib.ammari@ens.fr, wjing@dma.ens.fr).}
\and
Elie Bretin
\thanks{\footnotesize Institut Camille Jordan, INSA de Lyon, 69621, Villeurbanne Cedex, France (elie.bretin@insa-lyon.fr).}
\and
Josselin Garnier
\thanks{\footnotesize Laboratoire de Probabilit\'es et Mod\`eles Al\'eatoires \& Laboratoire Jacques-Louis Lions, Universit\'e Paris VII, 75205 Paris Cedex 13, France (garnier@math.jussieu.fr).}
\and  Wenjia Jing\footnotemark[2] \and Hyeonbae Kang
\thanks{\footnotesize Department of Mathematics, Inha University, Incheon, 402-751, Korea (hbkang@inha.ac.kr).}
\and
Abdul Wahab
\thanks{\footnotesize Department of Mathematics, COMSATS Institute of Information Technology, 47040, Wah Cantt., Pakistan (wahab@ciitwah.edu.pk).}
}

\maketitle

\begin{abstract}
The focus of this work is on rigorous mathematical analysis of the
topological derivative based detection algorithms for the
localization of an elastic inclusion of vanishing characteristic
size. A filtered quadratic misfit is considered and the
performance of the topological derivative imaging functional
resulting therefrom is analyzed. Our analysis reveals that the
imaging functional may not attain its maximum at the location of
the inclusion. Moreover, the resolution of the image is below the
diffraction limit. Both phenomena are due to the coupling of
pressure and shear waves propagating with different wave speeds
and polarization directions. A novel imaging functional based on
the weighted Helmholtz decomposition of the topological derivative
is, therefore, introduced. It is thereby substantiated that the
maximum of the imaging functional is attained at the location of
the inclusion and the resolution is enhanced and it proves to be
the diffraction limit. Finally, we investigate the stability of
the proposed imaging functionals with respect to measurement and
medium noises.
\end{abstract}



\section{Introduction}\label{sec:intro}

We consider the inverse problem of identifying the location  of a
small elastic inclusion  in a homogeneous isotropic background
medium from  boundary measurements. The main motivations of this
work are Non-Destructive Testing (NDT) of elastic structures for
material impurities \cite{Dominguez1}, exploration geophysics
\cite{Aki}, and medical diagnosis, in particular, for detection of
potential tumors of diminishing size \cite{Yuan}.

The long standing problem of anomaly detection has been addressed
using a variety of techniques including small volume expansion
methods \cite{AK-Pol, AK-book04}, MUSIC type algorithms
\cite{Direct} and time-reversal techniques \cite{TrElastic,
AGHL-Tr}. The focus of the present study is on the topological
derivative based anomaly detection algorithms for elasticity. As
shown in \cite{AGJK-Top}, in anti-plane elasticity, the
topological derivative based imaging functional performs well and
is robust with respect to noise and sparse or limited view
measurements. The objective of this work is to extend this concept
to the general case of linear isotropic elasticity. The analysis
is much more delicate in the general case than in the scalar case
because of the coupling between the shear and pressure waves.

The concept of topological derivative (TD), initially proposed for
shape optimization in \cite{Eschenauer, soko, Cea},  has been
recently applied to the imaging of small anomalies, see for
instance, \cite{Dominguez1, Dominguez2, GuzinaBon, Guzina,
Hintermuller, Hintermuller2, Masmoudi} and references therein.
However, its use in the context of imaging has been heuristic and
lacks mathematical justifications, notwithstanding its usefulness.

In a prior work \cite{AGJK-Top}, acoustic anomaly detection algorithms based on the concept
of TD are analyzed and their performance is compared with different detection techniques.
 Moreover, a stability and resolution analysis is carried out in the presence of measurement and medium noises.

The aim of this work is to analyze the ability of the TD based
sensitivity framework for detecting elastic inclusions of
vanishing characteristic size. Precisely, our goal  is threefold:
(i) to perform a rigorous mathematical analysis of the TD  based
imaging; (2) to design a modified imaging framework based on the
analysis. In the case of a density contrast, the modified
framework yields a topological derivative based imaging
functional, {\it i.e.}, deriving from the topological derivative
of a discrepancy functional. However, in the case where the Lam\'e
coefficients of the small inclusion are different from those of
the background medium, the modified functional is rather of a
Kirchhoff type. It is based on the correlations between,
separately,  the shear and compressional parts of the
backpropagation of the data and those of the background solution.
It can not be derived as the topological derivative of a
discrepancy functional; and (3) to investigate the stability of
the proposed imaging functionals with respect to measurement and
medium noises.


In order to put this work in a proper context, we emphasize some
of its significant achievements. A trial inclusion is created in
the background medium at a given search location.
Then, a discrepancy functional is considered (c.f. Section \ref{sec:TD}), which is
the elastic counterpart of the filtered quadratic misfit proposed in \cite{AGJK-Top}. The search points that minimize the discrepancy between measured data and
the fitted data are then sought for. In order to find its minima,
the misfit is expanded using the asymptotic expansions due to the
perturbation  of the displacement field in the presence of an
inclusion versus its characteristic size. The first order term in
the expansion is then referred to as TD of the misfit (c.f.
Section \ref{sec:TD:top}) which synthesizes its sensitivity
relative to the insertion of an inclusion at a given search
location. We show that its maximum, which corresponds to the point
at which the insertion of the inclusion maximally decreases the
misfit, may not be at the location of the true inclusion (c.f.
Section \ref{sec:TD:sensitivity}). Further, it is revealed that
its resolution is low due to the coupling of pressure and shear
wave modes having different wave speeds and polarization
directions. Nevertheless, the coupling terms responsible for this
degeneracy can be canceled out using a modified imaging framework.
A weighed imaging functional is defined using the concept of a
weighted Helmholtz decomposition, initially proposed in
\cite{TrElastic} for time reversal imaging of extended  elastic
sources. It is proved that the modified detection algorithm
provides a resolution limit of the order of half a wavelength,
indeed, as the new functional behaves as the square of the
imaginary part of a pressure or shear Green function (c.f. Section
\ref{sec:WTD:sensitivity}). For simplicity, we restrict ourselves
to the study of two particular situations when we have only a
density contrast or an elasticity contrast. In order to cater to
various applications, we provide explicit results for the
canonical cases of circular and spherical inclusions. It is also
important to note that the formulae of the TD based functionals
are explicit in terms of the incident wave and the free space
fundamental solution instead of the Green function in the bounded
domain with imposed boundary conditions. This is in contrast with
the prior results, see for instance, \cite{Guzina}. Albeit a
Neumann boundary condition is imposed on the displacement field,
the results of this paper extend to the problem with Dirichlet
boundary conditions. A stability analysis of the TD based imaging
functionals was also missing in the literature. In this paper we
carry out a detailed stability analysis of the proposed imaging
functionals with respect to both measurement and medium noises.

The rest of this paper is organized as follows: In Section
\ref{sec:mathform}, we introduce some notation and present the
asymptotic expansions due to the perturbation of the displacement
field in the presence of small inclusions. Section \ref{sec:TD} is
devoted to the study of TD imaging functional resulting from the
expansion of the filtered quadratic misfit with respect to the
size of the inclusion. As discussed in Section
\ref{sec:TD:sensitivity}, the resolution in TD imaging framework
is not optimal. Therefore, a modified imaging framework is
established in Section \ref{sec:WTD}. The sensitivity analysis of
the modified framework is presented in Section
\ref{sec:WTD:sensitivity}. Sections \ref{sec:ssmeas} and
\ref{sec:ssmedium}  are devoted to the stability analysis with
respect to measurement and medium noises, respectively. The paper
is concluded in Section \ref{sec:conclu}.

\section{Mathematical formulation}\label{sec:mathform}

This section is devoted to preliminaries, notation and assumptions
used in rest of this paper. We also recall a few fundamental
results related to small volume asymptotic expansions of the
displacement field due to the presence of a penetrable inclusion
with respect to the size of the inclusion, which will be essential
in the sequel.

\subsection{Preliminaries and Notations}\label{sec:mathform:not}

Consider a homogeneous isotropic elastic material occupying a bounded domain
$\Omega\subset\R^d$, for $d=2$ or $3$, with connected Lipschitz boundary $\partial\Omega$.
Let the Lam\'e (compressional and shear) parameters of  $\Omega$ be $\lambda_0$ and $\mu_0$
(respectively) in the absence of any inclusion and $\rho_0>0$ be the (constant) volume density
of the background. Let $D\subset\Omega$ be an elastic inclusion  with Lam\'e parameters $\lambda_1$,
 $\mu_1$ and density $\rho_1>0$.   Suppose that $D$ is given by
\begin{equation}\label{D-Parametrization}
D:=\delta B+\bz_a
\end{equation}
where $B$ is a bounded Lipschitz domain in $\RR^d$ containing the
origin and $\bz_a$ represents the location of the inclusion $D$.
The small parameter $\delta$ represents  the characteristic size
of the diameter of $D$. Moreover, we assume that $D$ is separated
apart from the boundary $\partial\Omega$, \emph{i.e.}, there
exists a constant $c_0>0$ such that
\begin{equation}\label{D-Omega-dist}
\inf_{\bx\in D}\rm{dist}(\bx,\partial\Omega)\geq c_0,
\end{equation}
where $ {\rm dist}$  denotes the distance. Further, it is assumed
that
\begin{equation}\label{Lame-Conditions}
d\lambda_m+2\mu_m >0,\quad \mu_m >0, \quad m \in \{0,1\}, \quad
(\lambda_0-\lambda_1)(\mu_0-\mu_1)\geq 0.
\end{equation}

Consider the following transmission problem with the Neumann
boundary condition:
\begin{equation}
\label{Prob_Trans}
\left\{
\begin{array}{ll}
\OL_{\lambda_0,\mu_0}\bu+\rho_0\omega^2\bu=0  & \text{in }\Dc,
\\\nm
\OL_{\lambda_1,\mu_1}\bu+\rho_1\omega^2\bu=0  & \text{in } D,
\\\nm
\bu{\big|_-}=\bu{\big|_+} & \text{on } \partial D,
\\\nm
\dis\frac{\partial\bu}{\partial\widetilde\nu}{\Big|_-}=\dis\frac{\partial\bu}{\partial\nu}{\Big|_+}
& \text{on } \partial D,
\\\nm
\dis\frac{\partial\bu}{\partial\nu}=\bg & \text{on }\partial\Omega,
\end{array}
\right.
\end{equation}
where $\omega>0$ is the angular frequency of the mechanical
oscillations, the linear elasticity system
$\OL_{\lambda_0,\mu_0}$ and the co-normal derivative
$\dis\frac{\partial}{\partial \nu}$, associated with parameters
$(\lambda_0,\mu_0)$ are defined by
\begin{equation}\label{LE_Def}
\OL_{\lambda_0,\mu_0} [\bw]:=
\mu_0\Delta\bw+(\lambda_0+\mu_0)\nabla\nabla\cdot\bw
\end{equation}
and
\begin{equation}\label{Conormal_Def}
\dis\frac{\partial \bw}{\partial \nu}:=\lambda_0(\nabla\cdot\bw)\bn+\mu_0(\nabla\bw^T+(\nabla\bw^T)^T)
\bn ,
\end{equation}
respectively. Here superscript $T$ indicates the transpose of a
matrix, $\bn$ represents the outward unit normal to $\partial D$,
and $\frac{\partial}{\partial \widetilde\nu }$ is the co-normal
derivative associated with $(\lambda_1,\mu_1)$. To insure
well-posedness, we assume that $\rho_0\omega^2$ is different from
the Neumann eigenvalues of the operator $-\OL_{\lambda_0,\mu_0}$
in $\left(L^2(\Omega)\right)^d$. Using the theory of collectively
compact operators (see, for instance, \cite[Appendix
A.3]{AK-book04}), one can show that for small $\delta$ the
transmission problem (\ref{Prob_Trans}) has a unique solution for
any $\bg \in \left(L^2(\partial \Omega)\right)^d$.

Throughout this work, for a domain $X$, notations $|_-$ and $|_+$
indicate respectively the limits from inside and from outside $X$
to its boundary $\partial X$,  $\delta_{ij}$ represents the
Kronecker's symbol and
$$
\alpha,\beta\in\{P,S\},
\qquad
i,j,k,l,i',j',k',l',p,q\in\{1,\cdots, d\},
\qquad
m\in \{0, 1 \},
$$
where $P$ and $S$ stand for pressure and shear parts,
respectively.

\subsubsection*{Statement of the Problem:} The problem under consideration is the following:

\emph{Given the displacement field $\bu$, the solution of the
Neumann problem \eqref{Prob_Trans} at the boundary
$\partial\Omega$,  identify the location $\bz_a$  of the inclusion
$D$ using a TD based sensitivity framework.}

\subsection{Asymptotic analysis and fundamental results}\label{sec:mathform:asymp}

Consider the fundamental solution
$\bGam_m(\bx,\by):=\bGam_m(\bx-\by)$ of the homogeneous
time-harmonic elastic wave equation in $\RR^d$ with parameters
$(\lambda_m,\mu_m,\rho_m)$, \emph{i.e.}, the solution to
\begin{equation} \label{Green_Eqn}
(\OL_{\lambda_m,\mu_m}+\rho_m\omega^2)\bGam_m(\bx-\by)=
\delta_{\by}(\bx)\ID,\qquad\forall \bx\in\RR^d,  \bx\neq \by ,
\end{equation}
subject to the \emph{Kupradze's} outgoing radiation conditions
\cite{Kup},  where $\delta_\by$ is the Dirac mass at $\by$ and
$\ID$ is the $d\times d$ identity matrix. Let
$c_{S}=\sqrt{\frac{\mu_0}{\rho_0}}$ and $c_P=\sqrt{
\frac{\lambda_0+ 2 \mu_0}{\rho_0}}$ be the background shear and
the pressure wave speeds respectively. Then  $\bGam_0$ is given by
\cite{Aki}
\begin{equation}\label{Green_fun}
\bGam_0(\bx)=\left\{\dis\frac{1}{\mu_0}\ID
G_S^\omega(\bx)-\frac{1}{\rho_0\omega^2} \DD_\bx\left[
G_P^\omega(\bx)-G_S^\omega(\bx)\right]\right\},\quad
\bx\in\R^d,\quad d=2,3,
\end{equation}
where the tensor $\DD_\bx$ is defined by  $$
\DD_\bx=\nabla_\bx\otimes\nabla_\bx=(\partial_{ij})_{i,j=1}^d,
$$ and the function $G_\alpha^\omega$ is the fundamental solution
to the Helmholtz operator, \emph{i.e.},
$$
(\Delta+\K^2_\alpha) G^\omega_\alpha(\bx)=\delta_{\bf 0}(\bx)\quad
\bx\in\RR^d, \bx\neq \mathbf{0},
$$
subject to the \emph{Sommerfeld's} outgoing radiation condition
$$
\left |\frac{\partial G^\omega_\alpha}{\partial\bn}-i\K_\alpha
G^\omega_\alpha\right|(\bx)=o(R^{1-d/2}), \qquad \bx\in\partial B
({\bf 0}, R),
$$
with $B ({\bf 0}, R)$ being the sphere of radius $R$ and center
the origin. Here $\partial_{ij} = \frac{\partial^2}{\partial x_i
\partial x_j}$, $\K_\alpha:=\frac{\omega}{c_\alpha}$ is the
wave-number, and $\frac{\partial}{\partial\bn}$ represents the
normal derivative.

The function $G^\omega_\alpha$ is given by
\begin{equation}
G^\omega_\alpha(\bx)= \left\{
\begin{array}{ll}
\dis - \frac{i}{4} H_0^{(1)}(\K_{\alpha}|\bx|), & d=2,
\\\nm
\dis - \frac{e^{i\K_\alpha|\bx|}}{4\pi |\bx|}, & d=3,
\end{array}
\right.
\end{equation}
where $H_n^{(1)}$ is the order $n$ Hankel function of first kind.

Note that  $\bGam_0$  can be decomposed into shear and pressure components \emph{i.e.}
\begin{equation}
\label{210}
\bGam_0(\bx)=\bGam_{0,S}(\bx)+\bGam_{0,P}(\bx),\quad\forall
\bx\in\RR^d,\quad \bx\neq \mathbf{0},
\end{equation}
where
\begin{equation} \label{211}
\bGam_{0,P}(\bx)= - \dis \frac{1}{\mu_0\K_S^2}\DD_\bx
G_P^\omega(\bx) \quad\text{and}\quad \bGam_{0,S}(\bx)=
\dis\frac{1}{\mu_0\K_S^2}(\K_S^2\ID+\DD_\bx)G_S^\omega(\bx).
\end{equation}
Note that $\nabla \cdot \bGam_{0,S} = \mathbf{0}$ and $\nabla
\times \bGam_{0,P} = \mathbf{0}$.

Let  us define the single layer potential $\S_\Omega^\omega$ associated with
 $(\OL_{\lambda_0,\mu_0}+\rho_0\omega^2)$ by
\begin{eqnarray}
\label{S-def} \S_\Omega^\omega[\mathbf{\Phi}](\bx)
&:=&\dis\int_{\partial\Omega}\bGam_0(\bx-\by)\mathbf{\Phi}(\by)d\sigma(\by),
\qquad \qquad \bx\in\RR^d, \label{D-def}
\end{eqnarray}
and the boundary integral operator $\Kcal_\Omega^\omega$ by
\begin{eqnarray}
\Kcal_\Omega^\omega[\mathbf{\Phi}](\bx) &:=& {\rm p.v.}
\dis\int_{\partial
\Omega}\frac{\partial}{\partial\nu_\by}\bGam_0(\bx-\by)\mathbf{\Phi}(\by)d\sigma(\by),\quad\text{a.e.
} \bx\in\partial\Omega
\end{eqnarray}
for any function $\mathbf{\Phi}\in \left(L^2(\partial\Omega
)\right)^d$, where ${\rm p.v.}$ stands for Cauchy principle value.

Let $(\Kcal^\omega_\Omega)^*$ be the adjoint operator of
$\Kcal^{-\omega}_\Omega$ on $\left(L^2(\partial\Omega )\right)^d$,
\emph{i.e.},
$$
(\Kcal_\Omega^\omega)^*[\mathbf{\Phi}](\bx) = {\rm p.v.}
\dis\int_{\partial \Omega}\frac{\partial}{\partial\nu_\bx}
\bGam_0(\bx-\by)\mathbf{\Phi}(\by)d\sigma(\by),\quad\text{a.e. }
\bx\in\partial\Omega .
$$
It is well known, see for instance  \cite[Section
3.4.3]{ammaribook}, that  the single layer potential,
$\S_\Omega^\omega$, enjoys the following jump conditions:
\begin{eqnarray}
\label{jumps}
\dis\frac{\partial(\S_\Omega^\omega[\mathbf{\Phi}])}{\partial
\nu}\Big|_{\pm}(\bx)=
\left(\pm\frac{1}{2}I+(\mathcal{K}_\Omega^\omega)^*\right)[\mathbf{\Phi}](\bx),\qquad\text{a.e.
} \bx\in\partial \Omega.
\end{eqnarray}

Let $\NN^\omega(\bx,\by)$, for all $\by\in\Omega$, be the Neumann
solution associated with $(\lambda_0,\mu_0,\rho_0)$ in $\Omega$,
\emph{i.e.},
\begin{equation}
\left\{
\begin{array}{ll}
(\OL_{\lambda_0,\mu_0}+\rho_0\omega^2)\NN^\omega (\bx,\by) = -
\delta_\by(\bx)\ID, &  \bx\in\Omega,\quad \bx\neq \by,
\\\nm
\dis\frac{\partial\NN^\omega }{\partial\nu}(\bx,\by)=0 &
\bx\in\partial\Omega.
\end{array}
\right.
\end{equation}
Then, by slightly modifying the proof for the case of zero
frequency in \cite{AK-Pol}, one can show that the following result
holds.
\begin{lem}\label{N-K-Gamma}
For all $\bx\in\partial\Omega$ and $\by\in\Omega$, we have
\begin{equation}
\left(\dis - \frac{1}{2}{I} + \Kcal_\Omega^\omega \right)
[\NN^\omega(\cdot,\by)](\bx)=\bGam_0(\bx - \by).
\end{equation}
\end{lem}

For $i,j\in\{1,\cdots,d\}$, let $\bv_{ij}$ be the solution to
\begin{equation}
\label{Wpq}
\left\{
\begin{array}{ll}
\OL_{\lambda_0,\mu_0}\bv_{ij}=0  & \text{in
}\RR^d\backslash\overline B,
\\\nm
\OL_{\lambda_1,\mu_1}\bv_{ij}=0  & \text{in } B,
\\\nm
\bv_{ij}{\big|_-}=\bv_{ij}{\big|_+} & \text{on } \partial B,
\\\nm
\dis\frac{\partial\bv_{ij}}{\partial\widetilde\nu}{\Big|_-}=\dis\frac{\partial\bv_{ij}}{\partial\nu}{\Big|_+}
& \text{on } \partial B,
\\\nm
\bv_{ij}(\bx)-x_i\be_j = O\left(|\bx|^{1-d}\right) &\text{as}\quad
|\bx|\to\infty,
\end{array}
\right.
\end{equation}
where $(\be_1,\cdots,\be_d)$ denotes the standard basis for
$\RR^d$. Then the elastic moment tensor (EMT)
$\MM:=\left(m_{ijpq}\right)^d_{i,j,p,q=1}$ associated with domain
$B$ and the Lam\'e parameters $(\lambda_0,\mu_0;\lambda_1,\mu_1)$
is defined by
\begin{equation}
m_{ijpq} =\int_{\partial
B}\left[\frac{\partial(x_p\be_q)}{\partial\widetilde\nu}
-\frac{\partial(x_p\be_q)}{\partial\nu}\right]\cdot \bv_{ij}\,d\sigma,
\end{equation}
see \cite{AK-Pol, AKNT-02}. In particular, for a circular or a
spherical inclusion, $\MM$ can be expressed as
\begin{equation}
\label{M-disk2} \MM = a \II_4 + b \ID \otimes \ID,
\end{equation}
or equivalently as
$$
m_{ijkl} = \frac{a}{2}(\delta_{ik} \delta_{jl} +
\delta_{il}\delta_{jk}) + b \delta_{ij} \delta_{kl},
$$
 for some constants $a$ and $b$ depending only on
$\lambda_0,\lambda_1,\mu_0,\mu_1$ and the space dimension $d$
\cite[Section 7.3.2]{ammaribook}. Here $\II_4$ is the identity
$4$-tensor. Note that for any $d\times d$ symmetric matrix
$\mathbf{A}$, $\II_4 (\mathbf{A})=\mathbf{A}$. Furthermore,
throughout this paper we make the assumption that $\mu_1 \geq
\mu_0$ and $\lambda_1 \geq \lambda_0$ in order to insure that the
constants $a$ and $b$ are positive.

It is known that EMT, $\MM$, has the following symmetry property:
\begin{equation}
m_{ijpq}=m_{pqij}=m_{jipq}=m_{ijqp},
\end{equation}
which allows us to identify  $\MM$ with a symmetric linear transformation on the space of symmetric
 $d\times d$ matrices. It also satisfies the positivity property (positive or
 negative definiteness) on the space of symmetric matrices \cite{AK-Pol, AKNT-02}.

Let $\bU$ be the background solution associated with
$(\lambda_0,\mu_0,\rho_0)$ in $\Omega$,  \emph{i.e.},
\begin{equation}\label{Background-Sol}
\left\{
\begin{array}{ll}
(\OL_{\lambda_0,\mu_0}+\rho_0\omega^2)\bU =0, & \text{on }\Omega,
\\\nm
\dis\frac{\partial\bU}{\partial\nu}=\bg & \text{on }\partial\Omega,
\end{array}
\right.
\end{equation}
Then, the following result can be obtained using analogous
arguments as in \cite{AK-Helm, AK-Pol}; see \cite{Direct}. Here
and throughout this paper
$$\mathbf{A} : \mathbf{B} = \sum_{i,j=1}^d a_{ij} b_{ij}$$
for matrices $\mathbf{A}=(a_{ij})_{i,j=1}^d$ and
$\mathbf{B}=(b_{ij})_{i,j=1}^d$.

\begin{thm}\label{Thm-Asymptotic}
Let $\bu$ be the solution to \eqref{Prob_Trans}, $\bU$ be the background solution defined by
 \eqref{Background-Sol} and $\rho_0\omega^2$ be different from the Neumann eigenvalues
  of the operator $-\OL_{\lambda_0,\mu_0}$ in $\left(L^2(\Omega)\right)^d$. Let $D$ be
   given by \eqref{D-Parametrization} and the conditions \eqref{D-Omega-dist} and  \eqref{Lame-Conditions}
   are satisfied. Then, for $\omega\delta\ll 1$, the following asymptotic expansion holds uniformly for all
    $\bx\in\partial\Omega$:
\begin{eqnarray}\label{Asymptotic-Exp}
&&\bu(\bx)-\bU(\bx) = - \delta^d\Big(\nabla\bU(\bz_a) : \MM(B)
\nabla_{\bz_a} \NN^\omega (\bx,\bz_a)
\\\nm
&&\qquad\qquad\qquad\qquad\qquad\qquad
+\omega^2(\rho_0-\rho_1)|B|\NN^\omega(\bx,\bz_a) \bU(\bz_a) \Big)
+O(\delta^{d+1}). \nonumber
\end{eqnarray}
\end{thm}

As a direct consequence of  expansion \eqref{Asymptotic-Exp} and
Lemma \ref{N-K-Gamma}, the following result holds.

\begin{cor}\label{Cor-Asymptotic}
Under the assumptions of Theorem \ref{Thm-Asymptotic}, we have
\begin{eqnarray}\label{Asymptotic-Exp-Gam}
&&\left(\dis\frac{1}{2}{I}- \Kcal_\Omega^\omega \right)[\bu
-\bU](\bx) = \delta^d\Big(\nabla\bU(\bz_a) :
\MM(B)\nabla_{\bz_a}\bGam_0(\bx - \bz_a)
\\\nm
&& \qquad\qquad\qquad\qquad\qquad\qquad\qquad
+\omega^2(\rho_0-\rho_1)|B|\bGam_0(\bx - \bz_a) \bU(\bz_a)
\Big)+O(\delta^{d+1})\qquad\qquad \nonumber
\end{eqnarray}
uniformly with respect to $\bx\in\partial\Omega$.
\end{cor}

\begin{rem} We have made use of the following conventions in \eqref{Asymptotic-Exp} and \eqref{Asymptotic-Exp-Gam}:
$$
\Big( \nabla\bU(\bz_a) :  \MM(B) \nabla_{\bz_a} \NN^\omega
(\bx,\bz_a) \Big)_{k} = \sum_{i,j=1}^d \Big( \partial_{i}
\bU_{j}(\bz_a) \sum_{p,q=1}^d m_{ijpq}\partial_{p}
\NN^\omega_{kq}(\bx,\bz_a) \Big),
$$
and
$$
\Big(\NN^\omega(\bx,\bz_a) \bU(\bz_a)\Big)_{k} = \sum_{i=1}^d
\NN_{ki}^\omega(\bx,\bz_a)  \bU_i(\bz_a).
$$
\end{rem}

\section{Imaging small inclusions using TD}\label{sec:TD}

In this section, we  consider a filtered quadratic misfit and
introduce a TD based imaging functional resulting therefrom and
analyze its performance when identifying true location $\bz_a$ of
the inclusion $D$.

For a search point $\bz^S$, let $\bu_{\bz^S}$ be the solution to
\eqref{Prob_Trans} in the presence of a trial inclusion
$D'=\delta' B'+\bz^S$ with parameters
$(\lambda'_1,\mu'_1,\rho_1')$, where $B'$ is chosen \emph{a
priori} and $\delta'$ is  small. Assume that
\begin{equation}\label{Lame-Prime-Conditions}
d\lambda'_1+2\mu'_1 >0,\quad \mu'_1 >0,\quad (\lambda_0-\lambda'_1)(\mu_0-\mu'_1)\geq 0.
\end{equation}
Consider the elastic counterpart of the filtered quadratic misfit
proposed by Ammari \emph{et al.} in \cite{AGJK-Top}, that is, the
following misfit:
\begin{equation}
\label{Ef-misfit}
\E_f[\bU](\bz^S)=\frac{1}{2}\int_{\partial\Omega}\left|\left(\dis\frac{1}{2}{I}-
\Kcal_\Omega^\omega
\right)[\bu_{\bz^S}-\bu_{\rm{meas}}](\bx)\right|^2d\sigma(\bx).
\end{equation}
As shown for  Helmholtz equations in \cite{AGJK-Top}, the
identification of the exact location of true inclusion using the
classical quadratic misfit
\begin{equation}\label{E-misfit}
\E[\bU](\bz^S)=\frac{1}{2}\int_{\partial\Omega}\big|(\bu_{\bz^S}-\bu_{\rm{meas}})(\bx)\big|^2
d\sigma(\bx)
\end{equation}
cannot be guaranteed and the post-processing of the data is
necessary. We show in the later part of this section that exact
identification can be achieved using filtered quadratic misfit
$\E_f$.

We emphasize that the post-processing compensates for the effects
of an imposed Neumann boundary condition on the displacement
field.


\subsection{Topological derivative of the filtered quadratic misfit}\label{sec:TD:top}

Analogously to Theorem \ref{Thm-Asymptotic}, the displacement
field $\bu_{\bz^S}$, in the presence of the trial inclusion at the
search location, can be expanded as
\begin{eqnarray}
&&\bu_{\bz^S}(x)-\bU(\bx) = - (\delta')^d\Big( \nabla\bU(\bz^S) :
\MM'(B') \nabla_{\bz^S}\NN^\omega(\bx,\bz^S) \nonumber
\\
&& \qquad \qquad \qquad \qquad \qquad \qquad +\omega^2(\rho_0 -
\rho'_1)|B'|\NN^\omega(\bx,\bz^S) \bU(\bz^S) \Big) +
O\left((\delta')^{d+1}\right),
\end{eqnarray}
for a small $\delta'>0$, where $\MM'(B')$  is the EMT associated
with the domain $B'$ and the  parameters
$(\lambda_0,\mu_0;\lambda'_1,\mu'_1)$. Following the arguments in
\cite{AGJK-Top}, we obtain, by using Corollary
\ref{Cor-Asymptotic} and the jump conditions \eqref{jumps}, that
\begin{eqnarray} \label{I}
\E_f[\bU](\bz^S)&=&
\frac{1}{2}\int_{\partial\Omega}\left|\left(\dis\frac{1}{2}{I}-
\Kcal_\Omega^\omega
\right)[\bU-\bu_{\rm{meas}}](\bx)\right|^2d\sigma(\bx) \nonumber
\\\nm
&& + (\delta')^d\Re e\left\{\nabla\bU(\bz^S):
\MM'(B')\nabla\bw(\bz^S)+\omega^2(\rho_0 - \rho'_1)|B'|\bU(\bz^S)
\cdot \bw(\bz^S)\right\} \nonumber
\\\nm
&&
+O\left((\delta\delta')^{d}\right)+O\left((\delta')^{2d}\right),
\end{eqnarray}
where the function $\bw$ is defined in terms of the measured data $\left(\bU-\bu_{\rm{meas}}\right)$ by
\begin{equation}\label{W-Def}
\bw(\bx)=\S_\Omega^\omega \bigg[\overline{\left(\dis\frac{1}{2}{I}
- \Kcal_\Omega^\omega \right) [\bu_{\rm{meas}} - \bU]}
\bigg](\bx),\quad \bx\in\Omega.
\end{equation}
The function $\bw$ corresponds to backpropagating inside $\Omega$
the boundary measurements of $\bU-\bu_{\rm{meas}}$. Substituting
\eqref{Asymptotic-Exp-Gam} in \eqref{W-Def}, we find that
\begin{eqnarray}
\bw(\bz^S)&=&\delta^d\Big(\nabla\overline{\bU}(\bz_a):\MM(B) \Big[
\int_{\partial\Omega}
\bGam_0(\bz^S-\bx)\nabla_{\bz_a}\overline{\bGam_0}(\bx-\bz_a)d\sigma(\bx)
\Big] \nonumber
\\\nm
&& +\omega^2(\rho_0 - \rho_1)|B|\Big[
\int_{\partial\Omega}\overline{\bGam_0}(\bx-\bz_a)\bGam_0(\bx-\bz^S)d\sigma(\bx)
\Big] \overline{\bU}(\bz_a) \Big) +O(\delta^{d+1}). \label{W-Exp}
\end{eqnarray}

\begin{defn}(Topological derivative of $\E_f$)
The TD imaging functional associated with $\E_f$ at a search point
$\bz^S\in\Omega$ is defined by
\begin{eqnarray}
\I_{\rm{TD}}[\bU](\bz^S)&:=&-\dis\frac{\partial\E_f[\bU](\bz^S)}{\partial(\delta')^d}\Big|_{(\delta')^d=0}
\label{TopDer-Ef-2}
\end{eqnarray}
\end{defn}

The functional $\I_{\rm{TD}}\left[\bU\right](\bz^S)$ at every
search point $\bz^S\in \Omega$ synthesizes the sensitivity of the
misfit $\E_f$ relative to the insertion of an elastic inclusion
$D'=\bz^S+\delta'B'$ at that point. The maximum of
$\I_{\rm{TD}}\left[\bU\right](\bz^S)$ corresponds to the point at
which the insertion of an inclusion centered at that point
maximally decreases the misfit $\E_f$. The location of the maximum
of $\I_{\rm{TD}}\left[\bU\right](\bz^S)$ is, therefore, a good
estimate of the location $\bz_a$ of the true inclusion, $D$, that
determines the measured field $\bu_{\rm{meas}}$. Note that from
(\ref{I}) it follows that
\begin{eqnarray}
\I_{\rm{TD}}[\bU](\bz^S)&=& - \Re e\Big\{\nabla\bU(\bz^S):
\MM'(B')\nabla\bw(\bz^S)+\omega^2(\rho_0 -
\rho'_1)|B'|\bU(\bz^S)\cdot\bw(\bz^S)\Big\}, \label{TopDer-Ef}
\end{eqnarray} where $\bw$ is given by (\ref{W-Exp}).

\subsection{Sensitivity analysis of TD}
\label{sec:TD:sensitivity} In this section, we explain why TD
imaging functional $\I_{\rm{TD}}$ may not attain its maximum at
the location $\bz_a$ of the true inclusion. Notice that the
functional $\I_{\rm{TD}}$ consists of two terms: a density
contrast term and an elasticity contrast term with background
material. For simplicity and for purely analysis sake, we consider
two special cases when we have only the density contrast or the
elasticity contrast with reference medium.

\subsubsection{Case I: Density contrast}
\label{sec:TD:sensitivity:I}
Suppose $\lambda_0 = \lambda_1$ and $\mu_0 = \mu_1$. In this case, the wave function $\bw$ satisfies
\begin{equation} \label{39}
\bw(\bz^S)\simeq \delta^d\left(\omega^2(\rho_0 - \rho_1)|B|\left[
\int_{\partial\Omega}\overline{\bGam_0}(\bx-\bz_a)\bGam_0(\bx-\bz^S)d\sigma(\bx)
\right] \overline{\bU}(\bz_a) \right).
\end{equation}
Consequently, the imaging functional $\I_{\rm{TD}}$ at
$\bz^S\in\Omega$ reduces to
\begin{eqnarray}
\I_{\rm{TD}}\left[\bU\right](\bz^S) \simeq C\,\omega^4\, \Re e
\bigg\{ \bU(\bz^S)\cdot \left[\left(\int_{\partial\Omega}
\overline{\bGam_0}(\bx-\bz_a)\bGam_0(\bx-\bz^S) d\sigma(\bx)
\right) \overline{\bU}(\bz_a) \right] \bigg\},
\end{eqnarray}
where
\begin{equation}
\label{C-const} C =  \delta^d  (\rho_0 - \rho'_1)(\rho_0 -
\rho_1)|B'||B|.
\end{equation}
Throughout this paper we assume that
$$(\rho_0 - \rho'_1)(\rho_0 -
\rho_1) \geq 0.$$

Let us recall the following estimates from \cite[Proposition
2.5]{TrElastic}, which hold as the distance between the points
$\bz^S$ and $\bz_a$ and the boundary $\partial \Omega$ goes to
infinity.
\begin{lem}(Helmholtz - Kirchhoff identities)\label{lem-HKI}
For $\bz^S,\bz_a\in\Omega$ far from the boundary $\partial
\Omega$, compared to the wavelength of the wave impinging upon
$\Omega$, we have
\begin{eqnarray*}
 \ds\int_{\partial\Omega} \overline{\bGam_{0,\alpha}}(\bx-\bz_a)\bGam_{0,\alpha} (\bx-\bz^S)d\sigma(\bx)
 &\simeq& -
 \dfrac{1}{c_\alpha \omega}\Im m \left\{ \bGam_{0,\alpha}(\bz^S-\bz_a) \right\},
 \\\nm
 \ds \int_{\partial\Omega}  \overline{\bGam_{0,\alpha}}(\bx-\bz_a) \bGam_{0,\beta} (\bx-\bz^S)d\sigma(\bx)
 &\simeq&  0,\quad \alpha\neq\beta.
\end{eqnarray*}
\end{lem}

Therefore, by virtue of (\ref{210}) and Lemma \ref{lem-HKI}, we
can easily get
\begin{eqnarray}
\I_{\rm{TD}}\left[\bU\right](\bz^S) \simeq  \ds -  C\,\omega^3\,
\Re e \,  \bigg\{\bU(\bz^S)\cdot\left[\Im m \left\{\dfrac{1}{c_P}
\bGam_{0,P}(\bz^S-\bz_a) + \dfrac{1}{c_S} \bGam_{0,S}(\bz^S-\bz_a)
\right\} \overline{\bU}(\bz_a) \right] \bigg\}.
\end{eqnarray}

Let $(\be_{\theta_1},\be_{\theta_2},\ldots, \be_{\theta_n})$ be
$n$ uniformly distributed directions over the unit disk or sphere,
and denote by $\bU_j^P$ and $\bU_j^S$ respectively the plane $P-$
and $S-$waves, that is,
\begin{equation}
\bU_j^P(\bx)  =  e^{i \K_P \bx \cdot \be_{\theta_j}}\be_{\theta_j}
\quad\text{and}\quad \bU_j^S(\bx) = e^{i \K_S \bx \cdot
\be_{\theta_j}} \be_{\theta_j}^{\perp} \label{plane-waves}
\end{equation}
for $d=2$. In three dimensions, we set $$\bU_{j,l}^S(\bx) = e^{i
\K_S \bx \cdot \be_{\theta_j}} \be_{\theta_j}^{\perp, l}, \quad
l=1,2,$$ where $(\be_{\theta_j}, \be_{\theta_j}^{\perp, 1},
\be_{\theta_j}^{\perp, 2})$ is an orthonormal basis of $\RR^3$.
For ease of notation, in three dimensions, $
\I_{\rm{TD}}[\bU_j^S](\bz^{S})$ denotes $\sum_{l=1}^2
\I_{\rm{TD}}[\bU_{j,l}^S](\bz^{S})$.

We have
\begin{equation} \label{planesum}
\frac{1}{n} \sum_{j=1}^n e^{i \K_\alpha \bx \cdot \be_{\theta_j}}
\simeq  - 4 (\frac{\pi}{\K_\alpha})^{d-2} \Im m \,
G^\omega_\alpha(\bx)
\end{equation} for large $n$; see, for instance, \cite{AGJK-Top}. The following proposition holds.
\begin{prop}
\label{prop-TD-caseI} Let $\bU^\alpha_j$ be defined in
\eqref{plane-waves}, where $j=1,2,\cdots, n$, for $n$ sufficiently
large. Then, for all $\bz^S\in\Omega$ far from $\partial\Omega$,
\begin{eqnarray}
&&\ds \dfrac{1}{n} \sum_{j=1}^{n} \I_{\rm{TD}}[\bU_j^P](\bz^{S})
\simeq 4 \mu_0 C \omega^3  (\frac{\pi}{\K_P})^{d-2}
(\frac{\K_S}{\K_P})^2 \Bigg[
\dfrac{1}{c_P} \left| \Im m \left\{ \bGam_{0,P}(\bz^S-\bz_a)\right\}\right|^2
\nonumber
\\
&& \qquad\qquad\qquad\qquad\qquad\qquad + \dfrac{1}{c_S} \Im m
\left\{\bGam_{0,P}(\bz^S-\bz_a)\right\} : \Im m
\left\{\bGam_{0,S}(\bz^S-\bz_a)\right\} \Bigg],
\end{eqnarray}
and
\begin{eqnarray}
&&\ds \dfrac{1}{n} \sum_{j=1}^{n} \I_{\rm{TD}}[\bU_j^S](\bz^{S})
\simeq
 4 \mu_0 C \omega^3  (\frac{\pi}{\K_S})^{d-2} \Bigg[
\dfrac{1}{c_S}\left| \Im m
\left\{\bGam_{0,S}(\bz^S-\bz_a)\right\}\right|^2 \nonumber
\\
&& \qquad\qquad\qquad\qquad\qquad\qquad + \dfrac{1}{c_P}\Im m
\left\{\bGam_{0,P}(\bz^S-\bz_a)\right\}  : \Im m
\left\{\bGam_{0,S}(\bz^S-\bz_a)\right\} \Bigg],
\end{eqnarray}
where $C$ is given by \eqref{C-const}.
\end{prop}
\begin{proof}
From (\ref{planesum}) it follows that
\begin{eqnarray}
\dfrac{1}{n} \sum_{j=1}^{n} e^{i \K_P \bx \cdot \be_{\theta_j}}
\be_{\theta_j} \otimes \be_{\theta_j} &\simeq&  4
(\frac{\pi}{\K_P})^{d-2} \Im m \left\{ \dfrac{1}{\K_P^2}
 \DD_\bx G^\omega_P(\bx) \right\} \nonumber
\\
&\simeq& - 4 \mu_0  (\frac{\pi}{\K_P})^{d-2} (\frac{\K_S}{\K_P})^2
\Im m  \left\{ \bGam_{0,P}(\bx)  \right\},
\label{approx-e-times-e}
\end{eqnarray}
and
\begin{eqnarray}
\dfrac{1}{n} \sum_{j=1}^{n} e^{i \K_S \bx \cdot\be_{\theta_j}}
\be_{\theta_j}^{\perp} \otimes\be_{\theta_j}^{\perp} &=&
\dfrac{1}{n} \sum_{j=1}^{n} e^{i \K_S \bx \cdot\be_{\theta_j}}
\left(\ID - \be_{\theta_j} \otimes\be_{\theta_j} \right) \nonumber
\\\nm
&\simeq& - 4 (\frac{\pi}{\K_S})^{d-2} \Im m \left\{ \left( \ID +
\dfrac{1}{\K_S^2} \DD_\bx \right) G^\omega_S(\bx) \right\}
\nonumber
\\\nm
&=&  - 4 \mu_0  (\frac{\pi}{\K_S})^{d-2} \Im m  \left\{
\bGam_{0,S}(\bx)  \right\}, \label{approx-eP-times-eP}
\end{eqnarray}
where the last equality comes from (\ref{211}). Note that, in
three dimensions, (\ref{approx-eP-times-eP}) is to be understood
as follows:
\begin{equation} \label{eq3D}
\dfrac{1}{n} \sum_{j=1}^{n} \sum_{l=1}^2 e^{i \K_S \bx
\cdot\be_{\theta_j}} \be_{\theta_j}^{\perp, l}
\otimes\be_{\theta_j}^{\perp, l} \simeq - 4 \mu_0
(\frac{\pi}{\K_S}) \Im m  \left\{ \bGam_{0,S}(\bx) \right\}.
\end{equation}
Then, using the definition of $\bU^P_j$ we compute imaging functional $\I_{\rm{TD}}$ for $n$ plane $P-$waves as
\begin{eqnarray*}
\ds \dfrac{1}{n} \sum_{j=1}^{n} \I_{\rm{TD}}[\bU_j^P](\bz^{S}) &=&
C \omega^4 \frac{1}{n} \sum_{j=1}^{n} \Re e \, \bU^{P}_j(\bz^S)
\cdot
 \left[ \int_{\partial\Omega}\overline{\bGam_0}(\bx-\bz_a)\bGam_0(\bx-\bz^S)d\sigma(\bx) \overline{\bU^{P}_j}(\bz_a) \right]
 \nonumber
\\\nm
&\simeq& - C  \omega^3  \dfrac{1}{n} \sum_{j=1}^{n} \Re e \, e^{i
\K_P (\bz^{S} - \bz_a) \cdot \be_{\theta_j}}\be_{\theta_j} \cdot
\Bigg[ \Im m \Big\{ \frac{1}{c_P}\bGam_{0,P}(\bz^S - \bz_a)
\nonumber
\\
&&\qquad\qquad\qquad\qquad\qquad\qquad\qquad\qquad +\frac{1}{c_S}
\bGam_{0,S}(\bz^S - \bz_a) \Big\}  \be_{\theta_j} \Bigg] \nonumber
\\\nm
&\simeq& - C \omega^3 \Re e \, \Bigg[\dfrac{1}{n} \sum_{j=1}^{n}
    e^{i \K_P (\bz^{S} - \bz_a)\cdot\be_{\theta_j}} \be_{\theta_j}
    \otimes \be_{\theta_j} \Bigg] :
\nonumber
\\
&& \qquad\qquad\qquad \Bigg[ \Im m
\left\{\frac{1}{c_P}\bGam_{0,P}(\bz^S - \bz_a) +\frac{1}{c_S}
\bGam_{0,S}(\bz^S - \bz_a) \right\} \Bigg]. \nonumber
\end{eqnarray*}
Here we used the fact that $\be_{\theta_j} \cdot \mathbf{A}
\be_{\theta_j} = \be_{\theta_j} \otimes \be_{\theta_j} :
\mathbf{A}$ for a matrix $\mathbf{A}$, which is easy to check.

Finally, exploiting the approximation \eqref{approx-e-times-e}, we conclude that
\begin{eqnarray*}
\ds \dfrac{1}{n} \sum_{j=1}^{n} \I_{\rm{TD}}[\bU_j^P](\bz^{S})
&\simeq& 4 \mu_0 C \omega^3  (\frac{\pi}{\K_P})^{d-2}
(\frac{\K_S}{\K_P})^2 \Bigg[ \dfrac{1}{c_P}\left| \Im m \left\{
\bGam_{0,P}(\bz^S-\bz_a)\right\}\right|^2 \nonumber
\\
&& \qquad\qquad\qquad + \dfrac{1}{c_S} \Im m
\left\{\bGam_{0,P}(\bz^S-\bz_a)\right\}  : \Im m
\left\{\bGam_{0,S}(\bz^S-\bz_a) \right\} \Bigg].
\end{eqnarray*}

Similarly, we can compute the imaging functional $\I_{\rm{TD}}$ for $n$ plane $S-$waves exploiting the approximation \eqref{approx-eP-times-eP}, as
\begin{eqnarray*}
\frac{1}{n} \sum_{j=1}^{n} \I_{\rm{TD}}[\bU_j^S](\bz^{S})
 &=&
 C \omega^4  \frac{1}{n} \sum_{j=1}^{n} \Re e\,  \bU^{S}_j(\bz^S)
\cdot\left[
\int_{\partial\Omega}\overline{\bGam_0}(\bx-\bz_a)\bGam_0(\bx-\bz^S)d\sigma(\bx)
\overline{\bU^{S}_j}(\bz_a) \right] \nonumber
\\
\nm
 &\simeq&
- C  \omega^3  \dfrac{1}{n} \sum_{j=1}^{n} \Re e\, e^{i \K_S
(\bz^{S} - \bz_a)\cdot\be_{\theta_j}}   \be_{\theta_j}^{\perp}
\cdot
 \Bigg[
 \Im m  \Big\{\dfrac{1}{c_P}\bGam_{0,P}(\bz^S - \bz_a)
 \nonumber
 \\
 &&
 \qquad\qquad\qquad\qquad\qquad\qquad\qquad\qquad
 + \dfrac{1}{c_S}\bGam_{0,S}(\bz^S - \bz_a) \Big\}\be_{\theta_j}^{\perp}
 \Bigg]
\nonumber
\end{eqnarray*}
\begin{eqnarray*}
 &\simeq& - C \omega^3  \Re e\, \Bigg[ \dfrac{1}{n} \sum_{j=1}^{n} e^{i
\K_S (\bz^{S} -\bz_a)\cdot\be_{\theta_j}} \be_{\theta}^{\perp}
\otimes\be_{\theta}^{\perp}\Bigg]: \nonumber
\\
&&
\qquad\qquad\qquad
 \Bigg[ \Im m \Big\{\dfrac{1}{c_P}\bGam_{0,P}(\bz^S - \bz_a)
 + \dfrac{1}{c_S}  \bGam_{0,S}(\bz^S - \bz_a) \Big\} \Bigg]
 \nonumber
\\
\nm
& \simeq &
 4 \mu_0 C \omega^3 (\frac{\pi}{\K_S})^{d-2}
\Bigg[ \dfrac{1}{c_S} \left|\Im m \left\{ \bGam_{0,S}(\bz^S-\bz_a)
\right\} \right|^2 \nonumber
\\
&& \qquad\qquad\qquad + \dfrac{1}{c_P} \Im m
\left\{\bGam_{0,P}(\bz^S-\bz_a)\right\}: \Im m
\left\{\bGam_{0,S}(\bz^S-\bz_a)\right\} \Bigg].
\end{eqnarray*}
In dimension 3, one should use (\ref{eq3D}) to get the desired
result. This completes the proof.
\end{proof}
From Proposition \ref{prop-TD-caseI}, it is not clear that the
imaging functional $\I_{\rm{TD}}$ attains its maximum at $\bz_a$.
Moreover,  for both $\ds\frac{1}{n} \sum_{j=1}^{n}
{\I}_{\rm{TD}}[\bU_j^S](\bz^{S})$ and $\ds\frac{1}{n}
\sum_{j=1}^{n} {\I}_{\rm{TD}}[\bU_j^P](\bz^{S})$ the resolution at
$\bz_a$ is not fine enough due to the presence of the term $ \Im m
\left\{\bGam_{0,P}(\bz^S-\bz_a)\right\}: \Im m
\left\{\bGam_{0,S}(\bz^S-\bz_a) \right\} . $ One way to cancel out
this term is to combine $\ds\frac{1}{n} \sum_{j=1}^{n}
{\I}_{\rm{TD}}[\bU_j^S](\bz^{S})$ and $\ds\frac{1}{n}
\sum_{j=1}^{n} {\I}_{\rm{TD}}[\bU_j^P](\bz^{S})$ as follows:
\begin{equation*}
\ds\frac{1}{n} \sum_{j=1}^{n} \bigg( c_S (\frac{\K_P}{\pi})^{d-2}
(\frac{\K_P}{\K_S})^2 {\I}_{\rm{TD}}[\bU_j^P](\bz^S) - c_P
(\frac{\K_S}{\pi})^{d-2} {\I}_{\rm{TD}}[\bU_j^S](\bz^S) \bigg).
\end{equation*}
However, one arrives at
\begin{eqnarray*}
&& \ds\frac{1}{n} \sum_{j=1}^{n} \bigg( c_S
(\frac{\K_P}{\pi})^{d-2} (\frac{\K_P}{\K_S})^2
{\I}_{\rm{TD}}[\bU_j^P](\bz^S) - c_P (\frac{\K_S}{\pi})^{d-2}
{\I}_{\rm{TD}}[\bU_j^S](\bz^S) \bigg) \nonumber
\\
&&
\qquad\qquad
\simeq
 4 \mu_0 C \omega^3  \bigg( \frac{c_S}{c_P}
\left|\Im m \left\{ \bGam_{0,P}(\bz^S-\bz_a) \right\} \right|^2 -
\frac{c_P}{c_S} \left|  \Im m \left\{
\bGam_{0,S}(\bz^S-\bz_a)\right\} \right|^2 \bigg),
\end{eqnarray*}
which is not a sum of positive terms and then can not guarantee
that the maximum of the obtained imaging functional is at the
location of the inclusion.


\subsubsection{Case II: Elasticity contrast}\label{sec:TD:sensitivity:II}
Suppose $\rho_0 = \rho_1$. Further, we assume for simplicity that
$\MM = \MM'(B') = \MM(B)$. From Lemma \ref{lem-HKI} we have
\begin{equation} \label{expl2} \begin{array}{lll}
 \ds\int_{\partial\Omega} \nabla_{\bz_a} \overline{\bGam_{0}}(\bx-\bz_a) \nabla_{\bz^S} \bGam_{0} (\bx-\bz^S)d\sigma(\bx)
 &\simeq&\ds
 - \dfrac{1}{c_S \omega}\Im m \left\{ \nabla_{\bz_a} \nabla_{\bz^S} \bGam_{0,S}(\bz^S-\bz_a) \right\} \\
 \nm && \ds -
 \dfrac{1}{c_P \omega}\Im m \left\{ \nabla_{\bz_a} \nabla_{\bz^S} \bGam_{0,P}(\bz^S-\bz_a)
 \right\}. \end{array}
\end{equation}
Then, using (\ref{W-Exp}) and (\ref{expl2}),
$\I_{\rm{TD}}\left[\bU\right](\bz^S)$ at $\bz^S\in\Omega$ becomes
\begin{eqnarray}
\I_{\rm{TD}}\left[\bU\right](\bz^S) &=& - \delta^d\, \Re e\,
\nabla \bU(\bz^S) : \MM \nabla \bw(\bz^S) \nonumber
\\
\nm &=& \delta^d\, \Re e\,  \nabla \bU(\bz^S) : \MM \left[
\int_{\partial\Omega}\nabla_{\bz_a}\overline{\bGam_0}(\bx-\bz_a)
\nabla_{\bz^S}\bGam_0(\bx-\bz^S)d\sigma(\bx) : \MM
\overline{\nabla \bU} (\bz_a) \right] \nonumber
\\
&\simeq&  \frac{\delta^d}{\omega}\, \Re e\,  \nabla \bU(\bz^S) :
\MM \Bigg[ \nabla^2 \Big( \Im m \Big\{ \widetilde{\bGam_{0}}(\bz^S
- \bz_a)\Big\} \Big) : \MM \overline{\nabla \bU}(\bz_a) \Bigg],
 \end{eqnarray}
where
\begin{equation}
\widetilde{\bGam_{0}}(\bz^S - \bz_a) = \frac{1}{c_P}
\bGam_{0,P}(\bz^S - \bz_a)+ \frac{1}{c_S} \bGam_{0,S}(\bz^S -
\bz_a).
\end{equation}

Let us define
\begin{eqnarray}
\label{J} J_{\alpha,\beta}(\bz^S) := \Big(\MM\Im
m\left[\big(\nabla^2\bGam_{0,\alpha}\big)(\bz^S-\bz_a)\right]\Big)
: \Big(\MM\Im m\left[
\big(\nabla^2\bGam_{0,\beta}\big)(\bz^S-\bz_a)\right]\Big)^T,
\end{eqnarray}
where $ \mathbb{A}^{T} = ({A}_{klij})$ if $\mathbb{A}$ is the
4-tensor given by $\mathbb{A} =({A}_{ijkl})$. Here $\mathbb{A} :
\mathbb{B} = \sum_{ijkl}A_{ijkl} B_{ijkl}$ for any $4$-tensors
$\mathbb{A}= ({A}_{ijkl})$ and $\mathbb{B}= ({B}_{ijkl})$.

The following result holds.
\begin{prop}
\label{prop-TD-caseII} Let $\bU^\alpha_j$ be defined in
\eqref{plane-waves}, where $j=1,2,\cdots, n$, for $n$ sufficiently
large. Let $J_{\alpha,\beta}$ be defined by \eqref{J}. Then, for
all $\bz^S\in\Omega$ far from $\partial\Omega$,
\begin{equation}
\ds \dfrac{1}{n} \sum_{j=1}^{n} \I_{\rm{TD}}[\bU_j^P](\bz^S)
\simeq
 4\delta^d \dfrac{\mu_0}{\omega }  (\frac{\pi}{\K_P})^{d-2} (\frac{\K_S}{\K_P})^2 \Big(\dfrac{1}{c_P}J_{P,P}(\bz^S)+\dfrac{1}{c_S}J_{S,P}(\bz^S)\Big)
\end{equation}
and
\begin{equation}
\ds \dfrac{1}{n} \sum_{j=1}^{n} \I_{\rm{TD}}[\bU_j^S](\bz^S)
\simeq
 4 \delta^d\dfrac{\mu_0}{\omega } (\frac{\pi}{\K_S})^{d-2} \Big(\dfrac{1}{c_S}J_{S,S}(\bz^S)+\dfrac{1}{c_P}J_{S,P}(\bz^S)\Big).
\end{equation}
\end{prop}
\begin{proof}
Let us compute $\I_{\rm{TD}}$ for $n$ plane $P-$waves, \emph{i.e.}
\begin{eqnarray}
\frac{1}{n} \sum_{j=1}^{n} \I_{\rm{TD}}[\bU_j^P](\bz^S) &=& \ds
\dfrac{\delta^d}{\omega}\frac{1}{n} \Re e\, \sum_{j=1}^{n}\nabla
\bU^{P}_j(\bz^S) :\MM \left[ \Im
m\left\{\big(\nabla^2\widetilde{\bGam_{0}}\big)(\bz^S - \bz_a)
\right\} : \MM \overline{\nabla \bU^{P}_j}(\bz_a) \right]
\nonumber
\\
\nm &\simeq&  \delta^d\frac{\omega}{c_P^2}\frac{1}{n} \Re e\,
\sum_{j=1}^{n} e^{i \K_P (\bz^S -
\bz_a)\cdot\be_{\theta_j}}\be_{\theta_j} \otimes \be_{\theta_j} :
\nonumber
\\
&& \qquad\qquad\qquad\qquad \MM \left( \Im m \left\{\nabla^2
\widetilde{\bGam_{0}}(\bz^S - \bz_a) \right\} : \MM\be_{\theta_j}
\otimes\be_{\theta_j}\right).
\end{eqnarray}
Equivalently,
\begin{eqnarray}
\frac{1}{n} \sum_{j=1}^{n} \I_{\rm{TD}}[\bU_j^P](\bz^S) &=&
\ds\delta^d\frac{\omega}{c_P^2}\frac{1}{n} \Re e\, \sum_{j=1}^{n}
e^{i\K_P (\bz^S - \bz_a)\cdot\be_{\theta_j}} \sum_{i,k,l,m=1}^d
\sum_{i',k',l',m'=1}^d {A}^{\theta_j}_{ik}\, m_{lmik} \nonumber
\\
&& \qquad\qquad \times\Im m  \left\{ \left(
\big(\partial_{li'}^2\widetilde{\bGam_{0}}\big)(\bz^S - \bz_a)
\right)_{mk'} \right\}m_{l'm'i'k'}\, {A}^{\theta_j}_{l'm'}
\end{eqnarray}
where the matrix  $\mathbf{A}^{\theta_j}
=({A}^{\theta_j}_{ik})_{ik}$ is defined as $ \mathbf{A}^{\theta_j}
:= \be_{\theta_j} \otimes \be_{\theta_j}. $ It follows that
\begin{eqnarray}
\frac{1}{n} \sum_{j=1}^{n} \I_{\rm{TD}}[\bU_j^P](\bz^S) &=& \ds
\delta^d  \Re e\, \sum_{i,k,l,m=1}^d \sum_{i',k',l',m'=1}^d
m_{lmik}\,m_{l'm'i'k'} \Im m  \left[ \left(
\big(\partial_{li'}^2\widetilde{\bGam_{0}}\big)(\bz^S - \bz_a)
\right)_{mk'} \right] \nonumber
\\
&& \qquad \qquad \ds \times\Bigg(\frac{\omega}{c_P^2} \frac{1}{n}
\sum_{j=1}^{n} e^{i \K_P (\bz^S - \bz_a)\cdot\be_{\theta_j}}
A^{\theta_j}_{ik} {A}^{\theta_j}_{l'm'}\Bigg).
\end{eqnarray}
Recall that for $n$ sufficiently large, we have from \eqref{approx-e-times-e}
\begin{equation*}
\ds\frac{1}{n} \sum_{j=1}^{n}  e^{i \K_P \bx \cdot \be_{\theta_j}}
\be_{\theta_j} \otimes \be_{\theta_j} \simeq - 4 \mu_0
(\frac{\pi}{\K_P})^{d-2} (\frac{\K_S}{\K_P})^2 \Im m \left\{
\bGam_{0,P}(\bx)  \right\}
\end{equation*}
(with the version (\ref{eq3D}) in dimension 3).
 Taking the Hessian of the previous approximation leads to
\begin{eqnarray}
\label{approx-A-Hessian} \ds\frac{1}{n} \sum_{j=1}^{n} e^{i \K_P
\bx\cdot\be_{\theta_j}} \be_{\theta_j} \otimes \be_{\theta_j}
\otimes \be_{\theta_j} \otimes \be_{\theta_j} &\simeq&  4 \mu_0
\frac{c_P^2}{\omega^2} (\frac{\pi}{\K_P})^{d-2}
(\frac{\K_S}{\K_P})^2 ~\Im m \left\{ \nabla^2 \bGam_{0,P}(\bx)
\right\} \nonumber\\
\nm &\simeq&  4 \mu_0 \frac{c_P^4}{\omega^2 c_S^2}
(\frac{\pi}{\K_P})^{d-2}~\Im m \left\{ \nabla^2 \bGam_{0,P}(\bx)
\right\}.
\end{eqnarray}
Then, by virtue of \eqref{approx-e-times-e} and \eqref{approx-A-Hessian}, we obtain
\begin{eqnarray}
\frac{1}{n} \sum_{j=1}^{n} \I_{\rm{TD}}[\bU_j^P](\bz^S) &\simeq&
 \delta^d \frac{4 \mu_0}{\omega} (\frac{\pi}{\K_P})^{d-2} (\frac{\K_S}{\K_P})^2 \sum_{i,k,l,m=1}^d \sum_{i',k',l',m'=1}^d
  m_{lmik}\,m_{l'm'i'k'}
\nonumber
\\
&& \qquad\qquad\qquad \ds\times \Im m  \left\{ \left(
\big(\partial_{li'}^2\widetilde{\bGam_{0}}\big)
(\bz^S-\bz_a)\right)_{mk'} \right\} \Im m \left\{
 \left( \big(\partial_{l'i}^2 \bGam_{0,P}\big) (\bz^S-\bz_a)\right)_{m'k} \right\} \nonumber
\\
\nm &\simeq& \ds \delta^d \frac{4 \mu_0}{\omega}
(\frac{\pi}{\K_P})^{d-2} (\frac{\K_S}{\K_P})^2
\sum_{i,k,i',k'=1}^d \Bigg( \sum_{l,m=1}^d m_{lmik} \Im m \left\{
\left( \big(\partial_{li'}^2 \widetilde{\bGam_{0}}\big)
(\bz^S-\bz_a)\right)_{mk'} \right\}\Bigg) \nonumber
\\
&& \qquad\qquad\qquad \times\ds\Bigg( \sum_{l',m'=1}^d
m_{l'm'i'k'} \Im m  \left\{ \left( \big(\partial_{l'i}^2
\bGam_{0,P}\big)(\bz^S-\bz_a)\right)_{m'k} \right\} \Bigg).
\nonumber
\end{eqnarray}
Therefore, by the definition \eqref{J} of $J_{\alpha,\beta}$, we conclude that
\begin{eqnarray*}
\frac{1}{n} \sum_{j=1}^{n} \I_{\rm{TD}}[\bU_j^P](\bz^S) &\simeq&
\ds \delta^d  \frac{4 \mu_0}{\omega} (\frac{\pi}{\K_P})^{d-2}
(\frac{\K_S}{\K_P})^2 \left( \MM\Im m\left\{ \nabla^2
\widetilde{\bGam_{0}}(\bz^S - \bz_a) \right\} \right)
\\
&& \qquad\qquad\qquad \ds : \Big( \MM\Im m\left\{  \nabla^2
\bGam_{0,P}(\bz^S - \bz_a)  \right\} \Big)^T
\\
&\simeq& \ds \delta^d  \frac{4 \mu_0}{\omega}
(\frac{\pi}{\K_P})^{d-2} (\frac{\K_S}{\K_P})^2
\left(\dfrac{1}{c_P}J_{P,P}(\bz^S)+\dfrac{1}{c_S}J_{S,P}(\bz^S)\right).
\end{eqnarray*}
Similarly, consider the case of plane $S-$waves and compute $\I_{\rm{TD}}$ for $n$ directions. We have
\begin{eqnarray}
\frac{1}{n}\sum_{j=1}^{n} \I_{\rm{TD}}[\bU_j^S](\bz^S) &=&
\dfrac{\delta^d}{\omega}\frac{1}{n} \Re e\, \sum_{j=1}^{n}\nabla
\bU^{S}_j(\bz^S):\MM \left(\Im m\left\{\big(\nabla^2
\widetilde{\bGam_{0}}\big)(\bz^S - \bz_a)\right\}
:\MM\overline{\nabla \bU^{S}_j}(\bz_a) \right) \nonumber
\\
\nm &\simeq& \delta^d\frac{\omega}{c_S^2}\frac{1}{n} \Re e\,
\sum_{j=1}^{n} e^{i \K_S (\bz^S - \bz_a)\cdot\be_{\theta_j}}
\be_{\theta_j}^{\perp} \otimes \be_{\theta_j} : \MM \nonumber
\\
&& \qquad\qquad\qquad\qquad \left( \Im m  \left\{ \big( \nabla^2
\widetilde{\bGam_{0}}\big)(\bz^S - \bz_a) \right\}
:\MM\,\be_{\theta_j}^{\perp} \otimes\be_{\theta_j} \right)
\nonumber
\\
\nm &\simeq& \delta^d\frac{\omega}{c_S^2}\frac{1}{n} \Re e\,
\sum_{j=1}^{n} e^{i \K_S (\bz^S - \bz_a)\cdot\be_{\theta_j}}
\sum_{i,k,l,m=1}^d\, \sum_{i',k',l',m'=1}^d {B}^{\theta_j}_{ik}\,
m_{lmik} \nonumber
\\
&& \qquad\qquad\qquad \times\Im m \left\{ \left(
\big(\partial_{li'}^2\widetilde{\bGam_{0}} \big)(\bz^S - \bz_a)
\right)_{mk'} \right\} m_{l'm'i'k'} {B}^{\theta_j}_{l'm'}
\end{eqnarray}
where the matrix  $\mathbf{B}^{\theta_j}
=({B}^{\theta_j}_{ik})_{ik}$ is defined as $\mathbf{B}^{\theta_j}
= \be_{\theta_j} \otimes \be_{\theta_j}^{\perp}$. It follows that
\begin{eqnarray}
\frac{1}{n} \sum_{j=1}^{n} \I_{\rm{TD}}[\bU_j^S](\bz^S) &=&
\delta^d \sum_{i,k,l,m=1}^d\,\sum_{i',k',l',m'=1}^d
m_{lmik}\,m_{l'm'i'k'} \Im m \left[\partial_{li'}^2
\left(\widetilde{\bGam_{0}}(\bz^S - \bz_a) \right)_{mk'} \right]
\nonumber
\\
\nm && \qquad\qquad\qquad\qquad \left(\frac{\omega}{c_S^2}
\frac{1}{n} \sum_{j=1}^{n}  e^{i \K_S (\bz^S -
\bz_a)\cdot\be_{\theta_j}} {B}^{\theta_j}_{ik}
{B}^{\theta_j}_{l'm'}\right).
\end{eqnarray}

Now, recall from \eqref{approx-eP-times-eP} that for $n$ sufficiently large, we have
$$
\frac{1}{n} \sum_{j=1}^{n} e^{i \K_S \bx \cdot\be_{\theta_j}}
\be_{\theta_j}^{\perp} \otimes\be_{\theta_j}^{\perp}\simeq - 4
\mu_0 (\frac{\pi}{\K_S})^{d-2} \Im m  \left\{  \bGam_{0,S}(\bx)
\right\} .
$$
Taking the Hessian of this approximation leads to
\begin{eqnarray}
\frac{1}{n} \sum_{j=1}^{n} e^{i \K_S \bx\cdot\be_{\theta_j}}
\be_{\theta_j} \otimes \be_{\theta_j}^{\perp} \otimes
\be_{\theta_j} \otimes \be_{\theta_j}^{\perp} \simeq  4\mu_0
\frac{c_S^2}{\omega^2} (\frac{\pi}{\K_S})^{d-2} \Im m \left\{
\nabla^2 \bGam_{0,S}(\bx) \right\},
\label{approx-eP-times-eP-Hessian}
\end{eqnarray}
where we have made use of the convention
$$
\left(\nabla^2 \bGam_{0,S}\right)_{ijkl} = \partial_{ik} \left(\bGam_{0,S}\right)_{jl}.
$$
Then, by using \eqref{approx-eP-times-eP}, \eqref{approx-eP-times-eP-Hessian} and the similar arguments as
in the case of $P-$waves, we arrive at
\begin{eqnarray*}
\frac{1}{n} \sum_{j=1}^{n} \I_{\rm{TD}}[\bU_j^S](\bz^S) &\simeq&
\delta^d \frac{4\mu_0}{\omega} (\frac{\pi}{\K_S})^{d-2}
\sum_{i,k,l,m=1}^d\, \sum_{i',k',l',m'=1}^d m_{lmik}\,m_{l'm'i'k'}
\nonumber \\ \nm && \quad \times \Im m  \left\{ \left(
\big(\partial_{li'}^2\widetilde{\bGam_{0}}\big)\right)_{mk'}(\bz^S
- \bz_a)\right\} \nonumber
\\
&& \qquad\qquad\qquad\qquad\qquad\qquad\qquad \times \Im m \left\{
 \left(\big(\partial_{l'i}^2 \bGam_{0,S}\big)\right)_{m'k}(\bz^S - \bz_a) \right\}
\nonumber
\\
\nm &\simeq& \delta^d \frac{4 \mu_0}{\omega}
(\frac{\pi}{\K_S})^{d-2} \left( \MM\,\Im m\left\{ \big(\nabla^2
\widetilde{\bGam_{0}}\big)(\bz^S - \bz_a)\right\} \right):
\nonumber
\\
&& \qquad\qquad\qquad\qquad \qquad\qquad\qquad \Big( \MM\,\Im
m\left\{ \big(\nabla^2  \bGam_{0,S}\big)(\bz^S - \bz_a) \right\}
\Big)^T \nonumber
\\
\nm &\simeq& \delta^d \frac{4 \mu_0}{\omega}
(\frac{\pi}{\K_S})^{d-2} \Big(\dfrac{1}{c_P}
J_{P,S}(\bz^S)+\dfrac{1}{c_S} J_{S,S}(\bz^S)\Big).
\end{eqnarray*}
This completes the proof.
\end{proof}
As observed in Section \ref{sec:TD:sensitivity:I}, Proposition
\ref{prop-TD-caseII} shows that the resolution of $\I_{\rm{TD}}$
deteriorates due to the presence of the coupling term
\begin{eqnarray}
J_{P,S}(\bz^S)= \Big( \MM\Im m\left\{
\big(\nabla^2\bGam_{0,S}\big) (\bz^S - \bz_a)  \right\} \Big) :
\Big( \MM\Im m\left\{  \big(\nabla^2 \bGam_{0,P}\big) (\bz^S -
\bz_a) \right\} \Big)^T.
\end{eqnarray}

\subsubsection{Summary}\label{sec:TD:sensitivity:sum}
To conclude, we summarize the results of this section below.
\begin{itemize}
\item[-] Propositions \ref{prop-TD-caseI} and \ref{prop-TD-caseII}
indicate that the imaging function $\I_{\rm{TD}}$ may not attain
its maximum at the true location, $\bz_a$, of the inclusion $D$.

\item[-] In both cases, the resolution of the localization of
elastic anomaly $D$ degenerates due to the presence of the
coupling terms $\Im m \left\{\bGam_{0,P}(\bz^S-\bz_a)\right\}  :
\Im m \left\{\bGam_{0,S}(\bz^S-\bz_a) \right\}$ and
$J_{P,S}(\bz^S)$, respectively. \item[-] In order to enhance
imaging resolution to its optimum and insure that the imaging
functional attains its maximum only at the location of the
inclusion, one must eradicate the coupling terms.
\end{itemize}

\section{Modified imaging framework}\label{sec:WTD}
In this section, in order to achieve a better localization and
resolution properties, we introduce a modified imaging framework
based on a weighted Helmholtz decomposition of the TD imaging
functional. We will show that the modified framework leads to both
a better localization (in the sense that the modified imaging
functional attains its maximum at the location of the inclusion)
and a better resolution than the classical TD based sensitivity
framework. It is worthwhile mentioning that the classical
framework performs quite well for the case of Helmholtz equation
\cite{AGJK-Top} and the resolution and localization deteriorations
are purely dependent on the elastic nature of the problem, that
is, due to the coupling of pressure and shear waves propagating
with different wave speeds and polarization directions.

It should be noted that in the case of a density contrast only,
the modified imaging functional is still a topological derivative
based one, {\it i.e.}, obtained as the topological derivative of a
discrepancy functional. This holds because of the nonconversion of
waves (from shear to compressional and vice versa) in the presence
of only a small inclusion with a contrast density. However, in the
presence of a small inclusion with different Lam\'e coefficients
with the background medium, there is a mode conversion; see, for
instance, \cite{conversion}. As a consequence, the modified
functional proposed here can not be written in such a case as the
topological derivative of a discrepancy functional. It is rather a
Kirchhoff-type imaging functional.

\subsection{Weighted imaging functional}\label{sec:WTD:functional}
Following \cite{TrElastic}, we introduce a weighted topological
derivative imaging functional ${\I}_{\rm{W}}$, and justify that it
provides a better localization of the inclusion $D$ than
$\I_{\rm{TD}}$. This new functional ${\I}_{\rm{W}}$ can be seen as
a correction based on a weighted Helmholtz decomposition of
$\I_{\rm{TD}}$. In fact, using the standard $L^2$-theory of the
Helmholtz decomposition (see, for instance, \cite{galdi}),  we
find that in the search domain the pressure and the shear
components of $\bw$, defined by \eqref{W-Def}, can be written as
\begin{equation}
\label{helmdecomps} \bw =  \nabla \times \psi_\bw + \nabla
\phi_\bw.
\end{equation}
We define respectively the Helmholtz decomposition operators
$\OH^P$ and $\OH^S$ by
\begin{equation}
\OH^P\left[ \bw \right] := \nabla \phi_{\bw }\quad \text{and}
\quad \OH^S\left[ \bw\right] := \nabla \times \psi_{\bw}.
\end{equation}
Actually, the decomposition $\bw = \OH^P\left[ \bw \right] +
\OH^S\left[ \bw \right]$ can be found by solving a Neumann problem
in the search domain \cite{galdi}. Then we multiply the components
of $\bw$ with $c_P$ and $c_S$, the background pressure and the
shear wave speeds respectively. Finally, we define
 ${\I}_{\rm{W}}$ by
\begin{eqnarray}
\label{tildeI} {\I}_{\rm{W}}\left[\bU\right] &=&\ds c_P \Re
e\left\{ - \nabla \OH^P[\bU] : \MM'(B')\nabla \OH^P[\bw]
+\omega^2\left(\dfrac{\rho'_1}{\rho_0}
-1\right)|B'|\OH^P[\bU]\cdot \OH^P[\bw]\right\} \nonumber
\\\nm
&+& c_S \Re e\left\{ - \nabla \OH^S[\bU] : \MM'(B')\nabla
\OH^S[\bw] +\omega^2\left(\dfrac{\rho'_1}{\rho_0}
-1\right)|B'|\OH^S[\bU]\cdot\OH^S[\bw]\right\}.
\end{eqnarray}
We  rigorously explain in the next section why  this new
functional should be better than imaging functional
$\I_{\rm{TD}}$.

\subsection{Sensitivity analysis of weighted imaging functional}
\label{sec:WTD:sensitivity} In this section, we explain why
imaging functional $\I_{\rm{W}}$ attains its maximum at the
location $\bz_a$ of the true inclusion with a better resolution
than $\I_{\rm{TD}}$.  In fact, as shown in the later part of this
section, $\I_{\rm{W}}$ behaves like the square of the imaginary
part of a pressure or a shear Green function
depending upon  the incident wave. Consequently, it provides a resolution of the order of half a wavelength. For simplicity, we once again consider special cases of only density contrast and only elasticity contrast.

\subsubsection{Case I: Density contrast}\label{sec:WTD:sensitivity:I}
Suppose $\lambda_0 = \lambda_1$ and $\mu_0 = \mu_1$. Recall that
in this case, the wave function $\bw$ is given by (\ref{39}). Note
that $\OH^\alpha[\bGam_0] = \bGam_{0,\alpha}, \alpha \in \{P,S\}.$
 Therefore, the imaging functional $\I_{\rm{W}}$ at
$\bz^S\in\Omega$ turns out to be
\begin{eqnarray}
 {\I}_{\rm{W}}\left[\bU\right](\bz^S)
 &=& C\, \omega^4
 \Re e\, \Bigg(
 c_P  \OH^P[\bU](\bz^S)\cdot
 \Big[ \Big(\int_{\partial\Omega} \overline{\bGam_0}(\bx-\bz_a)\bGam_{0,P}(\bx-\bz^S) d\sigma(\bx)
  \Big) \overline{\bU}(\bz_a)
 \Big]
 \nonumber
 \\
 &&
 + c_S  \OH^S[\bU](\bz^S)\cdot
 \Big[ \Big( \int_{\partial\Omega} \overline{\bGam_0}(\bx-\bz_a)\bGam_{0,S}(\bx-\bz^S) d\sigma(\bx)
  \Big)\overline{\bU}(\bz_a) \Big] \Bigg).
\end{eqnarray}

By using Lemma \ref{lem-HKI}, we can easily get
\begin{eqnarray}
{\I}_{\rm{W}}\left[\bU\right](\bz^S) &\simeq& \ds - C \,\omega^3
\Re e\, \Bigg( \OH^P[\bU](\bz^S)\cdot \Big[ \Im m
\left\{\bGam_{0,P}(\bz^S-\bz_a)\right\} \overline{\bU}(\bz_a)
\Big] \nonumber
\\
&& \qquad \qquad \quad +\OH^S[\bU](\bz^S)\cdot \Big[ \Im
m\left\{\bGam_{0,S}(\bz^S-\bz_a)\right\} \overline{\bU}
(\bz_a)\Big] \Bigg). \label{45}
\end{eqnarray}

Consider $n$ uniformly distributed directions
$(\be_{\theta_1},\be_{\theta_2},\ldots, \be_{\theta_n})$ on the
unit disk or sphere for $n$ sufficiently large. Then, the
following proposition holds.
\begin{prop}
\label{prop-WTD-caseI} Let $\bU^\alpha_j$ be defined in
\eqref{plane-waves}, where $j=1,2,\cdots, n$, for $n$ sufficiently
large. Then, for all $\bz^S\in\Omega$ far from $\partial\Omega$,
\begin{equation} \label{eqgf1}
\ds \dfrac{1}{n} \sum_{j=1}^{n} \I_{\rm{W}}[\bU_j^P](\bz^S) \simeq
 4 \mu_0 C \omega^3  (\frac{\pi}{\K_P})^{d-2} (\frac{\K_S}{\K_P})^2 \left|  \Im m
 \left\{ \bGam_{0,P}(\bz^S - \bz_a) \right\}
\right|^2,
\end{equation}
and
\begin{equation} \label{eqgf2}
\ds \dfrac{1}{n} \sum_{j=1}^{n} \I_{\rm{W}}[\bU_j^S](\bz^S) \simeq
4 \mu_0 C \omega^3 (\frac{\pi}{\K_S})^{d-2}  \left|  \Im m \left\{
\bGam_{0,S}(\bz^S - \bz_a) \right\} \right|^2,
\end{equation}
where $C$ is given by \eqref{C-const}.
\end{prop}

\begin{proof}
By using similar arguments as in Proposition \ref{prop-TD-caseI}
and (\ref{45}), we show that the weighted imaging functional
${\I}_{\rm{W}}$ for $n$ plane $P-$waves is given by
\begin{eqnarray*}
\dfrac{1}{n} \sum_{j=1}^{n} {\I}_{\rm{W}}[\bU_j^P](\bz^S) &=& -
C\, \omega^3 \dfrac{1}{n} \Re e\, \sum_{j=1}^{n}\bU^{P}_j(\bz^S)
\cdot \left[
 \Im m \left\{ \bGam_{0,P}(\bz^S - \bz_a) \right\}
\overline{\bU^{P}_j}(\bz_a) \right] \nonumber
\\\nm
&\simeq& - C \omega^3  \dfrac{1}{n}  \Re e\, \sum_{j=1}^{n} e^{i
\K_P (\bz^S - \bz_a).\be_{\theta_j}}\be_{\theta_j} \cdot \left[
\Im m \left\{ \bGam_{0,P}(\bz^S - \bz_a) \right\}\be_{\theta_j}
\right] \nonumber
\\\nm
&\simeq& 4 \mu_0 C \omega^3  (\frac{\pi}{\K_P})^{d-2}
(\frac{\K_S}{\K_P})^2 \left|  \Im m \left\{ \bGam_{0,P}(\bz^S -
\bz_a) \right\} \right|^2.
\end{eqnarray*}
For $n$ plane $S-$waves
\begin{eqnarray*}
\dfrac{1}{n} \sum_{j=1}^{n} {\I}_{\rm{W}}[\bU_j^S](\bz^S) &=& - C
\omega^3 \dfrac{1}{n} \sum_{j=1}^{n}\bU^{S}_j(\bz^S) \cdot\left[
 \Im m \left\{ \bGam_{0,S}(\bz^S - \bz_a) \right\}
\bU^{S}_j(\bz_a) \right] \nonumber
\\
\nm &\simeq& - C  \omega^3  \frac{1}{n} \sum_{j=1}^{n} e^{i \K_S
(\bz^S - \bz_a)\cdot\be_{\theta_j}}\be_{\theta_j}^{\perp} \cdot
\left[ \Im m  \left\{ \bGam_{0,S}(\bz^S - \bz_a) \right\}
\be^{\perp}_{\theta_j} \right] \nonumber
\\
\nm
&\simeq&
 4 \mu_0 C \omega^3 (\frac{\pi}{\K_S})^{d-2}
 \left|  \Im m \left\{ \bGam_{0,S}(\bz^S - \bz_a) \right\}
 \right|^2,
\end{eqnarray*}
where one should use the version (\ref{eq3D}) in dimension 3.
\end{proof}
Proposition \ref{prop-WTD-caseI} shows that ${\I}_{\rm{W}}$,
attains its maximum at $\bz_a$ (see Figure \ref{fig_max1}) and the
coupling term $ \Im m \left\{\bGam_{0,P}(\bz^S-\bz_a)\right\}: \Im
m \left\{\bGam_{0,S} (\bz^S-\bz_a)\right\}$, responsible for the
decreased resolution in ${\I}_{\rm{TD}}$, is absent. Moreover, the
resolution using weighted imaging functional ${\I}_{\rm{W}}$ is
the Rayleigh one, that is, restricted by the diffraction limit of
half a wavelength of the wave impinging upon $\Omega$, thanks to
the term $\left|\Im m  \left\{ \bGam_{0,\alpha}(\bz^S - \bz_a)
\right\}\right|^2$. Finally, it is worth mentioning that
${\I}_{\rm{W}}$ is a topological derivative based imaging
functional. In fact, it is the topological derivative of the
discrepancy functional $c_S \E_f[\bU^S] + c_P \E_f[\bU^P]$, where
$\bU^S$ is an $S$-plane wave and $\bU^P$ is a $P$-plane wave.

\begin{figure}[htp]
  \centering
  \subfigure{\includegraphics[width=6cm]{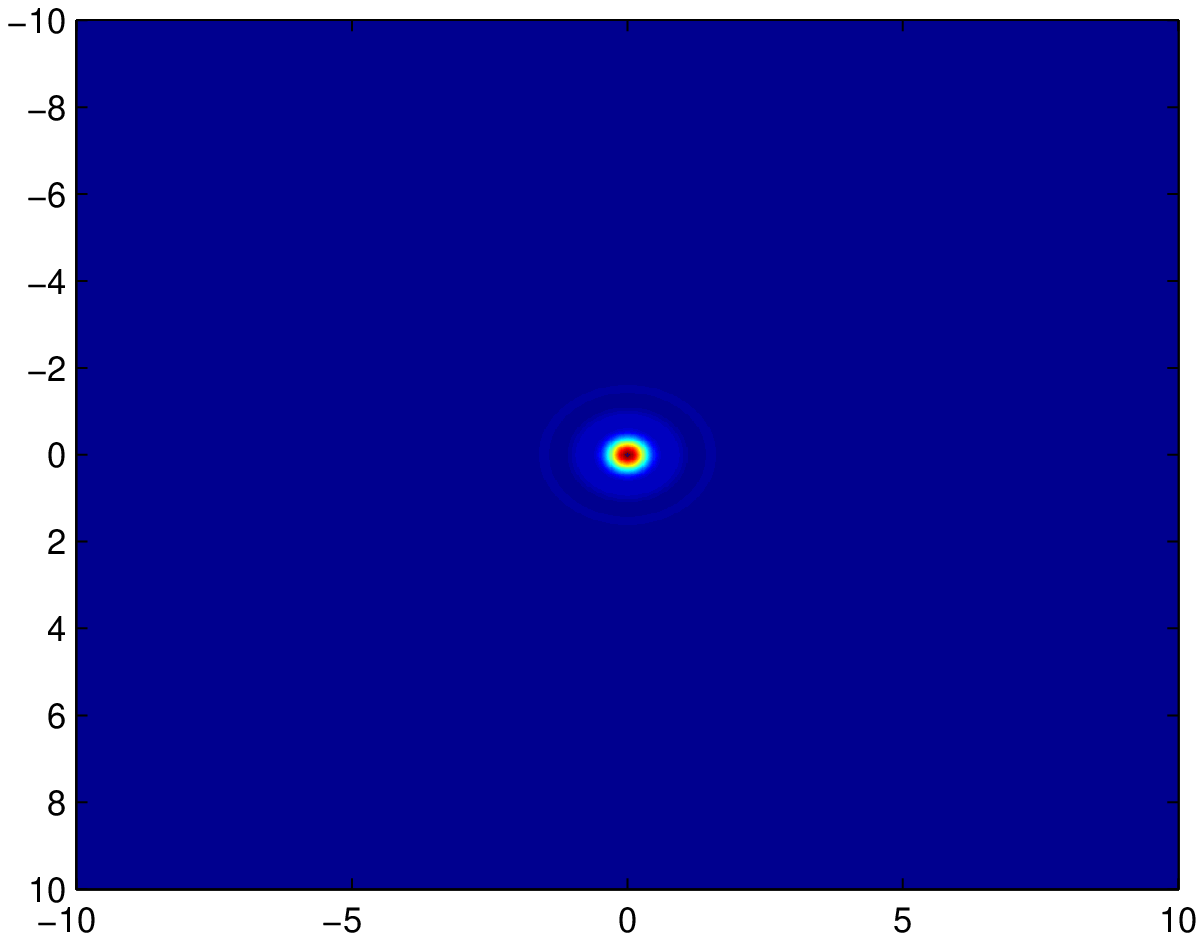}}
  \subfigure{\includegraphics[width=6cm]{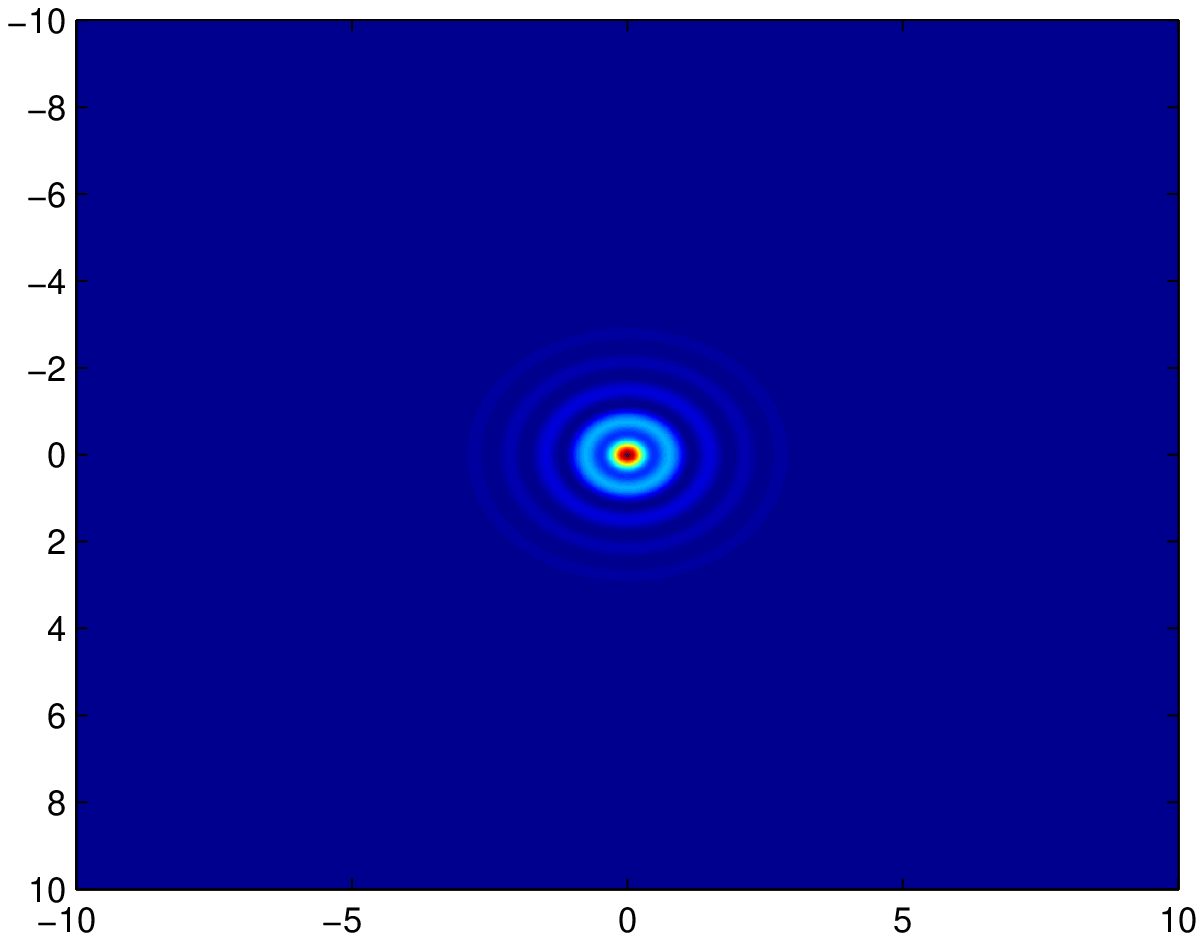}}
  \caption{ Typical plots of $\left|  \Im m
\left\{ \bGam_{0,S}(\bz^S - \bz_a) \right\} \right|^2$ (on the
left) and $\left|  \Im m \left\{ \bGam_{0,P}(\bz^S - \bz_a)
\right\} \right|^2$ (on the right) for $\bz_a =\bm{0}$ and
$c_P/c_S= \sqrt{11}$.}
  \label{fig_max1}
\end{figure}

\subsubsection{Case II: Elasticity contrast}\label{sec:WTD:sensitivity:II}
Suppose $\rho_0 = \rho_1$ and assume for simplicity that  $\MM =
\MM'(B') = \MM(B)$. Then, the weighted  imaging functional
${\I}_{\rm{W}}$ reduces to
\begin{eqnarray}
\I_{\rm{W}}(\bz^S) &=& - \delta^d \bigg[c_P \nabla
\OH^P[\bU(\bz^S)] :\MM \nabla \OH^P[\bw(\bz^S)] + c_S \nabla
\OH^S[\bU(\bz^S)] : \MM \nabla \OH^S \bw(\bz^S)]\bigg] \nonumber
\\
\nm &=& -  \delta^d \Bigg[ c_P \nabla\OH^P[\bU(\bz^S)]:\MM
\bigg(\int_{\partial\Omega}\nabla_{\bz_a}\overline{\bGam_0}(\bx-
\bz_a) \nabla_{\bz^S}\bGam_{0,P}(\bx-\bz^S)d\sigma(\bx)
:\MM\overline{\nabla \bU}(\bz_a) \bigg)\nonumber
\\
&& +c_S \nabla\OH^S[\bU(\bz^S)] :\MM
\bigg(\int_{\partial\Omega}\nabla_{\bz_a}\overline{\bGam_0}(\bx-\bz_a)
\nabla_{\bz^S}\bGam_{0,S}(\bx-\bz^S)d\sigma(\bx) : \MM
\overline{\nabla \bU}(\bz_a) \bigg)\Bigg] \nonumber
\\
\nm &=& - \delta^d \Bigg[ \nabla  \OH^P[\bU(\bz^S)] :\MM \Big( \Im
m \left\{ \big(\nabla^2 \bGam_{0,P}\big)(\bz^S - \bz_a)\right\}
:\MM\overline{\nabla \bU}(\bz_a)\Big) \nonumber
\\
&& +\nabla  \OH^S[\bU(\bz^S)] :\MM  \Big(\Im m \left\{
\big(\nabla^2 \bGam_{0,S}\big)(\bz^S - \bz_a)\right\} :\MM
\overline{\nabla \bU}(\bz_a)\Big) \Bigg].
\end{eqnarray}

We observed in Section \ref{sec:TD:sensitivity:II} that the
resolution of $\I_{\rm{TD}}$ is compromised because of the
coupling term $J_{S,P}(\bz^S)$. We can cancel out this term by
using the weighted imaging functional $\I_{\rm{W}}$. For example,
using analogous arguments as in Proposition \ref{prop-TD-caseII},
we can easily prove the following result.
\begin{prop}
\label{prop-WTD-caseII} Let $\bU^\alpha_j$ be defined in
\eqref{plane-waves}, where $j=1,2,\cdots, n$, for $n$ sufficiently
large. Let $J_{\alpha,\beta}$ be defined by \eqref{J}. Then, for
all $\bz^S\in\Omega$ far from $\partial\Omega$,
\begin{equation}
\ds \dfrac{1}{n} \sum_{j=1}^{n} \I_{\rm{W}}[\bU_j^\alpha](\bz^S)
\simeq 4\delta^d\dfrac{\mu_0}{\omega}
(\frac{\pi}{\K_\alpha})^{d-2} (\frac{\K_S}{\K_\alpha})^2
J_{\alpha,\alpha}(\bz^S), \quad \alpha \in \{ P,S\}.
\end{equation}
\end{prop}
It can be established that $\I_{\rm{W}}$ attains its maximum at
$\bz^S = \bz_a$. Consider, for example, the canonical case of a
circular or spherical inclusion. The following propositions hold.
\begin{prop}
\label{lem-Jpp} Let $D$ be a disk or a sphere. Then for all search
points $\bz^S\in\Omega$,
\begin{eqnarray}
J_{P,P}(\bz^S) &=&  a^2\Big|   \nabla^2 \big( \Im m\,
\bGam_{0,P}\big)(\bz^S - \bz_a)   \Big|^2 + 2ab \Big|\Delta
\big(\Im m \, \bGam_{0,P} \big)(\bz^S - \bz_a) \Big|^2 \nonumber
\\
\nm && + b^2 \Big|\Delta\, {\rm Tr}\big( \Im m\, \bGam_{0,P} \big)
(\bz^S - \bz_a)\Big|^2,
\end{eqnarray}
where ${\rm Tr}$ represents the trace operator and the constants
$a$ and $b$ are defined in \eqref{M-disk2}.
\end{prop}

\begin{proof} Since \begin{eqnarray}
\left(\nabla^2 \bGam_{0,P}\right)_{ijkl} = \partial_{ik}
\left(\bGam_{0,P}\right)_{jl},
\end{eqnarray}
it follows from (\ref{M-disk2}) that
\begin{eqnarray}
\left( \MM \nabla^2 \bGam_{0,P} \right)_{ijkl} &=& \ds \sum_{p,q}
m_{ijpq} \big(\nabla^2 \bGam_{0,P} \big)_{pqkl}
\\ \nm &=&\ds\frac{a}{2} \left( \partial_{ik}\left(\bGam_{0,P}\right)_{jl} +
\partial_{jk}\left(\bGam_{0,P}\right)_{il} \right) + b \sum_{q=1}^d
\partial_{qk}\left(\bGam_{0,P}\right)_{ql}  \delta_{ij} \nonumber
\\
\nm &=& \frac{a}{2} \partial_k\bigg(\big( \nabla
\bGam_{0,P}\be_{l} \big)_{ij} + \big( \nabla \bGam_{0,P}\be_{l}
\big)^{T}_{ij}\bigg) + b\partial_k \divg \bigg(\big(
\bGam_{0,P}\be_l \big) \bigg)  \delta_{ij},
\label{M-hessian-component}
\end{eqnarray}
where $\be_l$ is the unit vector in the direction $x_l$.

Now, since $\bGam_{0,p}\be_{l}$  is a $P-$wave, its rotational
part vanishes and the gradient is symmetric, \emph{i.e.},
\begin{equation}
\curl (\bGam_{0,P}\be_{l}) = 0 \quad\text{and}\quad \left(\nabla
\bGam_{0,P}\be_{l}\right)_{ij} = \left(\nabla
\bGam_{0,P}\be_{l}\right)_{ji} = \left(\nabla \bGam_{0,P}\be_{l}
\right)^{T}_{ij}. \label{Gamma-Gradient-Sym}
\end{equation}
Consequently,
\begin{equation}
\nabla \divg \Big(\left(\bGam_{0,P}\be_l \right) \Big) = \curl
\Big( \curl \left(  \bGam_{0,P}\be_l \right)\Big) + \Delta\Big(
\bGam_{0,P}\be_l \Big) = \Delta \Big(\bGam_{0,P}\be_l \Big),
\end{equation}
which, together with \eqref{M-hessian-component} and \eqref{Gamma-Gradient-Sym}, implies
\begin{equation}
\MM \nabla^2 \bGam_{0,P}= a\,\nabla^2 \bGam_{0,P} + b\,\ID \otimes
\Delta \bGam_{0,P}. \label{M-Hessian-final}
\end{equation}
Moreover, by the definition of $\bGam_{0,P}$, its Hessian,
$\nabla^2 \bGam_{0,P}$, is also symmetric. Indeed,
\begin{equation}
\Big(\nabla^2 \bGam_{0,P}\Big)^T_{ijkl} = \partial_{ki} \left(
\bGam_{0,P}\right)_{lj} = - \frac{\mu_0}{\K_S^2}
\partial_{kijl}G_P^{\omega} = \Big(\nabla^2
\bGam_{0,P}\Big)_{ijkl}. \label{Gamma-Hessian-Sym}
\end{equation}

Therefore, by virtue of \eqref{M-Hessian-final} and \eqref{Gamma-Hessian-Sym}, $J_{P,P}$ can be rewritten as
\begin{eqnarray}
&& J_{P,P}(\bz^S)
 =
\Big(a  \Im m \{ \big(\nabla^2  \bGam_{0,P}\big) (\bz^S - \bz_a)
\} + b \ID \otimes  \Im m \{ \big(\Delta \bGam_{0,P} \big)(\bz^S -
\bz_a) \} \Big) \nonumber
\\
&& \qquad :\Big(a \Im m \{\big(\nabla^2 \bGam_{0,P}\big)(\bz^S -
\bz_a) \} + b  \Im m \{ \big(\Delta \bGam_{0,P}\big)(\bz^S -
\bz_a) \} \otimes \ID  \Big). \label{Jpp-2nd-Last}
\end{eqnarray}
Finally, we observe that
\begin{eqnarray}
\Big(\nabla^2 \Im m \left\{ \bGam_{0,P} \right\} \Big)
:\Big(\nabla^2 \Im m \left\{\bGam_{0,P} \right\}\Big)^{T} =
\Big|\nabla^2 \Im m \left\{\bGam_{0,P} \right\}\Big|^2, \label{A}
\end{eqnarray}
\begin{eqnarray}
\nabla^2 \Im m \{\bGam_{0,P} \} :\Big( \ID \otimes \Delta \Im m
\{\bGam_{0,P}\} \Big) &=& \nabla^2 \Im m \{\bGam_{0,P}\} :
\Big(\Delta \Im m \{\bGam_{0,P}\} \otimes \ID \Big) \nonumber
\\
\nm &=& \sum_{i,j,k,l=1}^d \Big(\Im m\, \big (\partial_{ik}
\bGam_{0,P}\big)_{jl}\Big)\delta_{ij} \Delta \Im m\,
\big(\bGam_{0,P}\big)_{kl} \nonumber
\\
\nm &=& \sum_{k,l=1}^d\Bigg(\sum_{i=1}^d \left(\Im m\, \big(
\partial_{ik}\bGam_{0,P} \big)_{il}\right) \Bigg) \Delta\Im m \, \big(\bGam_{0,P}\big)_{kl}
\nonumber
\\
\nm &=& \sum_{k,l=1}^d\Big(\Delta\Im m \, \big(
\bGam_{0,P}\big)_{kl} \Big)^2 \nonumber
\\
\nm &=& \Big|\Delta \Im m \, \{\bGam_{0,P}\}\Big|^2, \label{B}
\end{eqnarray}
and
\begin{eqnarray}
\Big( \ID \otimes \Delta \Im m \{\bGam_{0,P}\} \Big) :\Big(\Delta
\Im m \{\bGam_{0,P}\}\otimes \ID  \Big) &=& \sum_{i,j,k,l=1}^d
\delta_{ij} \Delta \Im m \, \big(\bGam_{0,P}\big)_{kl} \delta_{kl}
\Delta \Im m \, \big(\bGam_{0,P}\big)_{ij} \nonumber
\\
\nm &=& \sum_{i,k=1}^d \Delta \Im m\,  \big( \bGam_{0,P}\big)_{kk}
\Delta \Im m \, \big(\bGam_{0,P}\big)_{ii} \nonumber
\\
\nm &=& \Big|\Delta\, {\rm Tr}(\Im m \, \{\bGam_{0,P}\})\Big|^2.
\label{C}
\end{eqnarray}
We arrive at the conclusion by substituting \eqref{A}, \eqref{B} and \eqref{C} in \eqref{Jpp-2nd-Last}.
\end{proof}

\begin{prop}
\label{lem-Jss} Let $D$ be a disk or a sphere. Then, for all
search points $\bz^S\in\Omega$,  \begin{eqnarray} J_{S,S}(\bz^S)
&=& \frac{a^2}{\mu_0^2} \Bigg[ \frac{1}{\K_S^4} \Big| \nabla^4 \Im
m\left\{ G_S^\omega(\bz^S-\bz_a)\right\}\Big|^2 +
\frac{(d-6)}{4}\Big| \nabla^2 \Im m \left\{ G_S^\omega(\bz^S -
\bz_a) \right\}\Big|^2 \nonumber
\\
&& + \dis\frac{\K_S^4}{4} \Big| \Im
m\left\{G_S^\omega(\bz^S-\bz_a)\right\}\Big|^2 \Bigg] \nonumber\\
&=& \frac{a^2}{\mu_0^2} \Bigg[ \frac{1}{\K_S^4} \sum_{ijkl, k\neq
l}  \Big| \partial_{ijkl} \Im m\left\{
G_S^\omega(\bz^S-\bz_a)\right\}\Big|^2 + \frac{(d-2)}{4}\Big|
\nabla^2 \Im m \left\{ G_S^\omega(\bz^S - \bz_a) \right\}\Big|^2
\nonumber
\\
&& + \dis\frac{\K_S^4}{4} \Big| \Im
m\left\{G_S^\omega(\bz^S-\bz_a)\right\}\Big|^2 \Bigg],
\end{eqnarray}
where $a$ is the constant as in \eqref{M-disk2}.
\end{prop}

\begin{proof}
As before, we have
\begin{eqnarray}
\bigg(\MM \nabla^2 \bGam_{0,S} \bigg)_{ijkl} &=& \frac{a}{2}\left(
\partial_{ik}\left(\bGam_{0,S}\right)_{jl} +
\partial_{jk}\left(\bGam_{0,S}\right)_{il} \right) +
b\,\partial_k \divg \bigg( \left(\bGam_{0,S}\be_l \right)\bigg)
\delta_{ij}
 \nonumber
 \\
 \nm
 &=&
 \frac{a}{2}\left( \partial_{ik}\left(\bGam_{0,S}\right)_{jl}
 + \partial_{jk}\left(\bGam_{0,S}\right)_{il}\right),
\label{M-Hessian-S-1st}
\end{eqnarray}
and
\begin{eqnarray}
\label{MT-Hessian-S-1st}
\Big(\MM \nabla^2 \bGam_{0,S}\Big)^T_{ijkl}
=
\frac{a}{2}\left( \partial_{ik}\left(\bGam_{0,S}\right)_{jl}
 + \partial_{il}\left(\bGam_{0,S}\right)_{jk}\right).
\end{eqnarray}
Here we have used the facts that $\bGam_{0,S}\be_{l}$ is a
$S-$wave and, $\bGam_{0,S}$ and its Hessian are symmetric,
\emph{i.e.},
\begin{equation}
\partial_{ik}\left(\bGam_{0,S}\right)_{jl}
= \partial_{ki}\left(\bGam_{0,S}\right)_{jl}
= \partial_{ik} \left(\bGam_{0,S}\right)_{lj}
= \partial_{ki} \left(\bGam_{0,S}\right)_{lj}.
\label{Gamm-S-Hessian-Sym}
\end{equation}

Substituting, \eqref{M-Hessian-S-1st} and \eqref{MT-Hessian-S-1st} in  \eqref{J},  we obtain
\begin{eqnarray}
J_{S,S}(\bz^S) &=& \frac{a^2}{4} \sum_{i,j,k,l=1}^d \Im m\left\{
\left( \big(\partial_{ik}
\bGam_{0,S}\big)(\bz^S-\bz_a)\right)_{jl} +
 \left( \big(\partial_{jk} \bGam_{0,S}\big)(\bz^S-\bz_a)\right)_{il}\right\}
\nonumber
\\
&& \qquad\qquad\qquad \times \Im m\left\{ \left(
\big(\partial_{ik}\bGam_{0,S}\big)(\bz^S-\bz_a)\right)_{jl} +
\left( \big(\partial_{il}
\bGam_{0,S}\big)(\bz^S-\bz_a)\right)_{jk} \right\} \nonumber
\\
&:=& \frac{a^2}{4} \Big(T_1(\bz^S) + 2T_2 (\bz^S)+
T_3(\bz^S)\Big), \label{J-I}
\end{eqnarray}
where
$$
\begin{cases}
T_1(\bz^S) &= \ds\sum_{i,j,k,l=1}^d \left(\Im m\left\{
\big(\partial_{i k}
\bGam_{0,S}\big)_{jl}(\bz^S-\bz_a)\right\}\right)
 \left(\Im
m\left\{ \big(\partial_{ik}
\bGam_{0,S}\big)_{jl}(\bz^S-\bz_a)\right\}\right),
\\
\nm T_2(\bz^S) &=  \ds\sum_{i,j,k,l=1}^d \left(\Im m\left\{
\big(\partial_{ik}
\bGam_{0,S}\big)_{jl}(\bz^S-\bz_a)\right\}\right)
 \left(\Im m\left\{ \big(\partial_{il} \bGam_{0,S}\big)_{jk}(\bz^S-\bz_a)\right\}\right),
\\
\nm T_3(\bz^S)&=  \ds\sum_{i,j,k,l=1}^d
 \left(\Im m\left\{ \big(\partial_{jk} \bGam_{0,S}\big)_{il}(\bz^S-\bz_a)\right\}\right)
 \left(\Im m\left\{ \big(\partial_{il} \bGam_{0,S}\big)_{jk}(\bz^S-\bz_a)\right\}\right).
\end{cases}
$$
Notice that
\begin{eqnarray*}
\Im m\left\{ \bGam_{0,S}(\bx)\right\} =
\dis\frac{1}{\mu_0\K_S^2}(\K_S^2\ID+\DD_\bx)\Im m\left\{
G_S^\omega(\bx)\right\},
\end{eqnarray*}
and $\Im m\left\{ G_S^\omega\right\}$ satisfies
\begin{eqnarray} \label{jkeq}
\Delta \Im m\left\{ G_S^\omega\right\} (\bz^S-\bz_a)  + \K_S^2 \Im
m\left\{ G_S^\omega\right\} (\bz^S-\bz_a) = 0 \quad \mbox{for }
\bz^S \neq \bz_a.
\end{eqnarray}
Therefore, the first term $T_1$ can be computed as follows
\begin{eqnarray*}
 T_1(\bz^S) &=&\Big|\nabla^2 \big(\Im m \bGam_{0,S}\big)(\bz^S-\bz_a)\Big|^2
 \\\nm
&=& \dis\frac{1}{\mu_0^2\K_S^4} \sum_{i,j,k,l=1}^d \Big[
\Big(\partial_{ijkl} \big(\Im m \, G_S^\omega \big)(\bz^S-\bz_a)
\Big)^2
+ \K_S^4\delta_{jl} \Big(\partial_{ik} \big(\Im m \, G_S^\omega
\big)(\bz^S-\bz_a) \Big)^2
\\
&& \qquad\qquad\qquad\qquad\quad + 2 \K_S^2 \delta_{jl}
\partial_{ik} \big(\Im m \, G_S^\omega \big)(\bz^S-\bz_a)
\partial_{ijkl} \big(\Im m  G_S^\omega\big)(\bz^S-\bz_a)
\Big].
\end{eqnarray*}
We also have
\begin{eqnarray*}
&&\sum_{i,j,k,l=1}^d 2 \delta_{jl} \partial_{ik} \big(\Im m \,
G_S^\omega\big)(\bz^S-\bz_a) \Big(\partial_{ijkl} \big( \Im m \,
G_S^\omega \big)(\bz^S-\bz_a) \Big)
\\
\nm && \qquad\qquad\qquad =2  \sum_{i,k=1}^d \Big(
\partial_{ik} \big( \Im m \,  G_S^\omega\big)(\bz^S-\bz_a)\Big)
\left(\partial_{ik} \sum_{l=1}^d \partial_{ll} \big(\Im m \,
G_S^\omega\big)(\bz^S-\bz_a) \right)
\\
\nm && \qquad\qquad\qquad =-2 \K_S^2 \sum_{i,k=1}^d\Big(
\partial_{ik} \big(\Im m \, G_S^\omega\big)(\bz^S-\bz_a) \Big)^2,
\end{eqnarray*}
and
\begin{eqnarray*}
 \sum_{i,j,k,l=1}^d \delta_{jl} \Big(\partial_{ik} \big(\Im
m \, G_S^\omega\big)(\bz^S-\bz_a) \Big)^2 = d
\sum_{i,k=1}^d\Big(\partial_{ik} \big(\Im m \,
G_S^\omega\big)(\bz^S-\bz_a) \Big)^2.
\end{eqnarray*}
Consequently, we have
\begin{equation}
\begin{array}{lll}
T_1(\bz^S) &=& \Big| \nabla^2 \big(\Im m
\bGam_{0,S}\big)(\bz^S-\bz_a) \Big|^2 = \frac{1}{\mu_0^2 \K_S^4}
\Big| \nabla^4 \big(\Im m\,
G_S^\omega\big)(\bz^S-\bz_a) \Big|^2 \\
\nm && + \frac{(d-2)}{\mu_0^2} \sum_{i,k=1}^d\Big(\partial_{ik}
\big(\Im m \, G_S^\omega\big)(\bz^S-\bz_a) \Big)^2. \label{T_1}
\end{array}
\end{equation}

Estimation of the term $T_2$ is quite similar. Indeed,
\begin{eqnarray*}
 T_2(\bz^S)
 &=&
 \dis\frac{1}{\mu_0^2\K_S^4} \sum_{i,j,k,l= 1}^d
 \Bigg[ \Big(\partial_{ijkl} \big(\Im m \, G_S^\omega\big)(\bz^S-\bz_a) \Big)^2
 \nonumber
 \\
 \nm
 &&
 \qquad\qquad\qquad\qquad
 + 2 \K_S^2 \delta_{jl} \partial_{ik} \big(
 \Im m\, G_S^\omega\big)(\bz^S-\bz_a)
 \partial_{ijkl}\big(\Im m \, G_S^\omega\big)(\bz^S-\bz_a)
 \nonumber
 \\
 \nm
&& \qquad \qquad \qquad \qquad + \K_S^4\delta_{jl}
\delta_{jk}\Big(\partial_{ik} \big(\Im m\,
G_S^\omega\big)(\bz^S-\bz_a) \Big)\Big(\partial_{il} \big(\Im m \,
G_S^\omega\big)(\bz^S-\bz_a) \Big) \Bigg].
\end{eqnarray*}
Finally, using
\begin{eqnarray*}
  \sum_{i,j,k,l=1}^d  \delta_{jl} \delta_{jk}
  \Big(\partial_{ik}  \big(\Im m \, G_S^\omega\big)(\bz^S-\bz_a) \Big)
  \Big(\partial_{i l} \big(\Im m \, G_S^\omega\big)(\bz^S-\bz_a) \Big)
  =
  \sum_{i,k=1}^d
  \Big(\partial_{ik} \big(\Im m \, G_S^\omega\big)(\bz^S-\bz_a)\Big)^2,
\end{eqnarray*}
we obtain that
\begin{eqnarray}
T_2(\bz^S) = \frac{1}{\mu_0^2 \K_S^4} \Big| \nabla^4 \big(\Im m \,
G_S^\omega\big)(\bz^S-\bz_a) \Big|^2 - \dis\frac{1}{\mu_0^2} \Big|
\nabla^2 \big(\Im m \, G_S^\omega\big)(\bz^S-\bz_a) \Big|^2 .
\label{T_2}
\end{eqnarray}
Similarly,
\begin{eqnarray*}
T_3(\bz^S)  &=& \dis\frac{1}{\mu_0^2\K_S^4} \sum_{i,j,k,l=1}^d
\Bigg[ \Big(\partial_{ijkl} \big(\Im m\,
G_S^\omega\big)(\bz^S-\bz_a) \Big)^2 \nonumber
\\
\nm && \qquad\qquad\qquad\qquad + 2 \K_S^2 \delta_{jl}
\partial_{ik}\big(\Im m \, G_S^\omega\big)(\bz^S-\bz_a)
\Big(\partial_{ijkl} \big(\Im m \,
G_S^\omega\big)(\bz^S-\bz_a)\Big)
\\
\nm && \qquad \qquad \qquad\qquad + \K_S^4\delta_{il} \delta_{jk}
\Big(\partial_{jk} \big(\Im m \, G_S^\omega\big)(\bz^S-\bz_a)
\Big) \Big(\partial_{il} \big(\Im m\,
G_S^\omega\big)(\bz^S-\bz_a)\Big) \Bigg].
\end{eqnarray*}
By virtue of
\begin{eqnarray*}
 && \sum_{i,j,k,l=1}^d \delta_{il} \delta_{jk}\left(\partial_{jk}
 \big(\Im m \, G_S^\omega\big)(\bz^S-\bz_a)\right)
 \left(\partial_{il} \big(\Im m \, G_S^\omega\big)(\bz^S-\bz_a)\right)
\nonumber
\\
\nm
&&
\qquad\qquad\qquad\qquad
 = \sum_{i,k=1}^d
 \Big( \partial_{kk} \big(\Im m \, G_S^\omega\big)(\bz^S-\bz_a) \Big) \Big( \partial_{ii}
 \big(\Im m \, G_S^\omega\big)(\bz^S-\bz_a)\Big)
 \nonumber
 \\
 \nm
&& \qquad\qquad\qquad\qquad =\K_S^4 \Big( \Im m \,
G_S^\omega(\bz^S-\bz_a) \Big)^2,
\end{eqnarray*}
we have
\begin{eqnarray}
T_3(\bz^S) &=& \frac{1}{\mu_0^2 \K_S^4} \Big| \nabla^4 \big(\Im m
\, G_S^\omega\big)(\bz^S-\bz_a)\Big|^2 -
 \dis\frac{2}{\mu_0^2}
 \Big| \nabla^2 \big(\Im m \, G_S^\omega\big)(\bz^S-\bz_a) \Big|^2
 \nonumber
 \\
 \nm
 &&
 +
 \dis\frac{\K_S^4}{\mu_0^2}
 \Big| \Im m \, G_S^\omega(\bz^S-\bz_a) \Big|^2.
 \label{T_3}
\end{eqnarray}
We conclude the proof by substituting \eqref{T_1}, \eqref{T_2} and
\eqref{T_3} in \eqref{J-I} and using again \eqref{jkeq}.
\end{proof}

Figure \ref{fig_max2} shows typical plots of $J_{\alpha,\alpha}$
for $\alpha \in \{ P,S\}$.
\begin{figure}[htp]
  \centering
  \subfigure{\includegraphics[width=6cm]{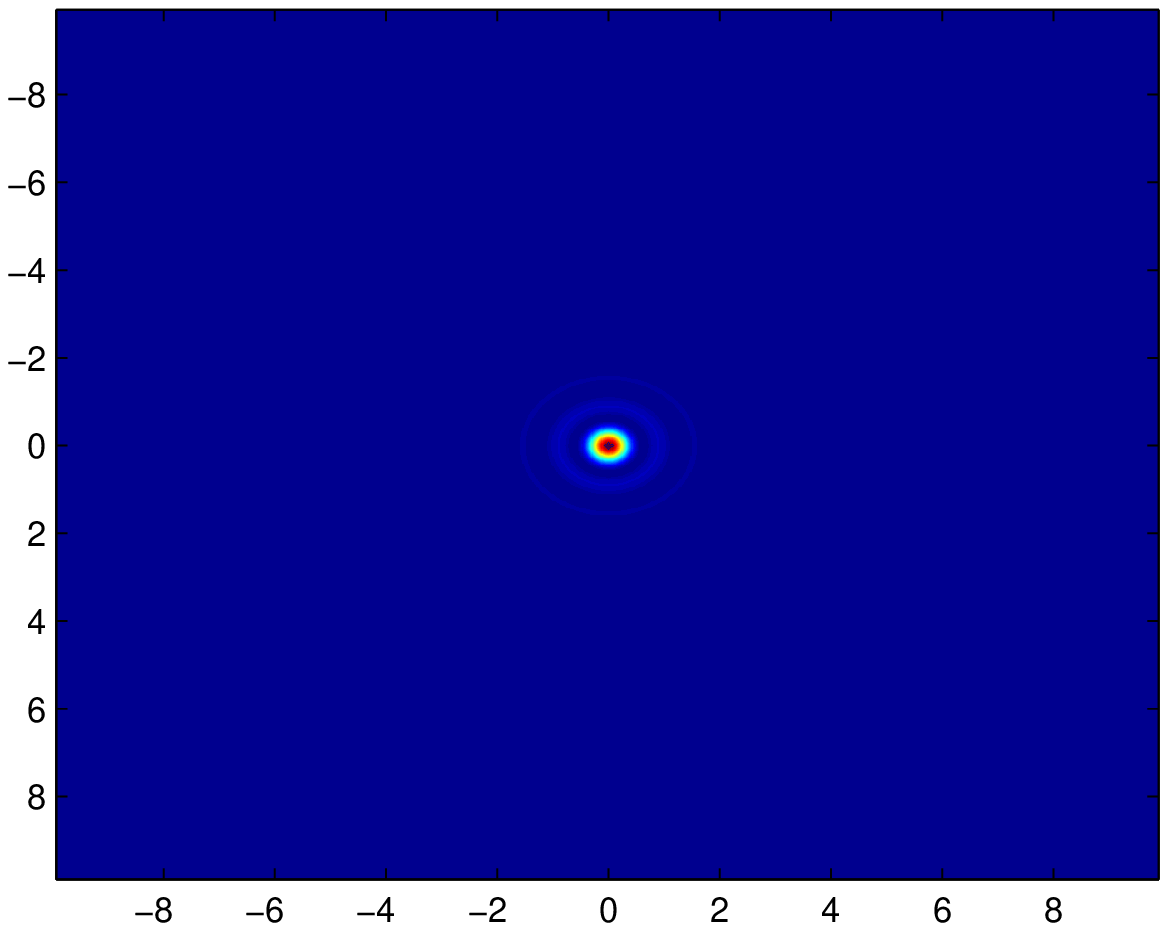}}
  \subfigure{\includegraphics[width=6cm]{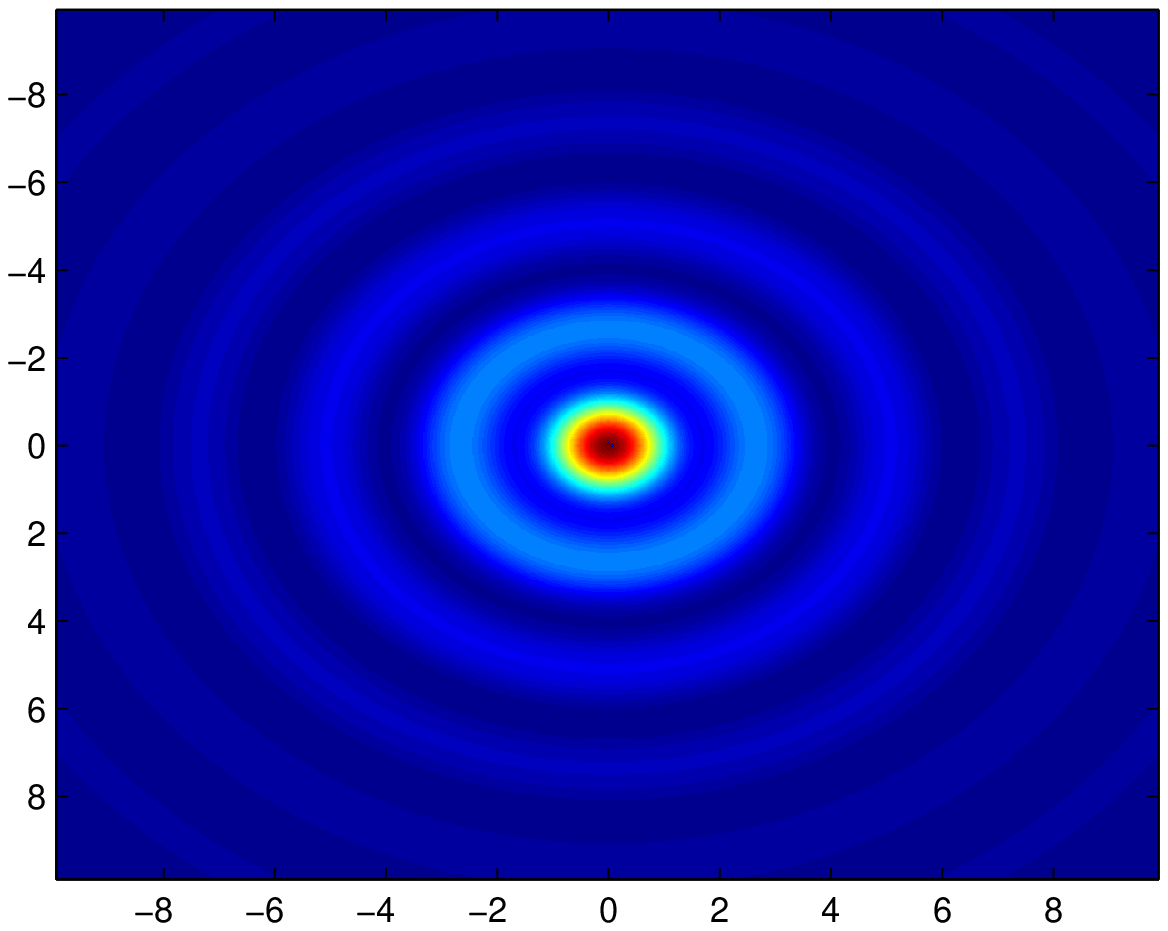}}
  \caption{ Typical plots of $J_{SS}$ (on the
right) and $J_{PP}$ (on the right) for $\bz_a =\bm{0}$ and
$c_P/c_S= \sqrt{11}$.}
  \label{fig_max2}
\end{figure}

\section{Statistical stability with measurement noise}
\label{sec:ssmeas}

Let $\bU^P_j$ and $\bU^S_j$ be as before. Let $\{\bU_j\}$ be plane
waves. Define
\begin{equation}
\IWF [\{\bU_j\}](\bz^S) = \frac{1}{n} \sum_{j=1}^n \I_{\rm W}
[\bU_j](\bz^S). \label{eq:IWFdef}
\end{equation}
In the previous section, we have analyzed the resolution of the
imaging functional $\IWF$ in the ideal situation where the
measurement $\bu_{\rm{meas}}$ is accurate. Here, we analyze how
the result will be modified when the measurement is corrupted by
noise.

\subsection{Measurement noise model}

We consider the simplest model for the measurement noise. Let
$\bu_{\rm{true}}$ be the accurate value of the  elastic
displacement field. The measurement $\bu_{\rm{meas}}$ is then
\begin{equation}
\bu_{\rm{meas}} (\bx) = \bu_{\rm{true}}(\bx) +
\bnu_{\rm{noise}}(\bx), \label{eq:umeasnoise}
\end{equation}
that is the accurate value corrupted by measurement noise modeled
as $\bnu_{\rm{noise}}(\bx)$, $\bx\in \partial \Omega$. Note that
$\bnu_{\rm{noise}}(\bx)$ is valued in $\CC^d, d = 2,3$.

Let $\EE$ denote the  expectation with respect to the statistics
of the measurement noise. We assume that
$\{\bnu_{\rm{noise}}(\bx), \bx \in \partial \Omega\}$ is mean zero
circular Gaussian and satisfies
\begin{equation}
\EE [\bnu_{\rm{noise}}(\by) \otimes
\overline{\bnu_{\rm{noise}}(\by')}] = \signoise^2
\delta_{\by}(\by') \ID. \label{eq:noimod}
\end{equation}
This means that firstly the measurement noises at different
locations on the boundary are uncorrelated; secondly, different
components of the measurement noise are uncorrelated, and thirdly
the real and imaginary parts are uncorrelated. Finally, the noise
has variance $\signoise^2$.

In the imaging functional $\IWF$, the elastic medium is probed by
multiple plane waves with different propagating directions, and
consequently multiple measurements are obtained at the boundary
accordingly. We assume that two measurements corresponding to two
different plane wave propagations are uncorrelated. Therefore, it
holds that
\begin{equation}
\EE [\bnu_{\rm{noise}}^j(\by) \otimes
\overline{\bnu_{\rm{noise}}^l(\by')}] = \signoise^2
\delta_{jl}\delta_{\by}(\by') \ID, \label{eq:twonoi}
\end{equation}
where $j$ and $l$ are labels for the measurements and
$\delta_{jl}$ is the Kronecker symbol.

\subsection{Propagation of measurement noise in the back-propagation step}

The measurement noise affects the topological derivative based
imaging functional through the back-propagation step which builds
the function $\bw$ in \eqref{W-Def}. Due to the noise, we have
\begin{equation}
\bw (\bx) = \S_\Omega^\omega
\bigg[\overline{\left(\dis\frac{1}{2}{I} - \Kcal_\Omega^\omega
\right) [\bU-\bu_{\rm{true}}-\bnu_{\rm{noise}}]} \bigg](\bx) =
\bw_{\rm{true}}(\bx) + \bw_{\rm{noise}}(\bx), \label{eq:bwdec}
\end{equation}
for $\bx\in\Omega$. Here, $\bw_{\rm{true}}$ is the result of
back-propagating only the accurate data while $\bw_{\rm{noise}}$
is that of back-propagating the measurement noise. In particular,
\begin{equation}
\bw_{\rm{noise}}(\bx) = -\S_\Omega^\omega
\bigg[\overline{\left(\dis\frac{1}{2}{I} - \Kcal_\Omega^\omega
\right) [\bnu_{\rm{noise}}]} \bigg](\bx), \quad \bx \in \Omega.
\label{eq:wnoi}
\end{equation}

To analyze the statistics of $\bw_{\rm{noise}}$, we proceed in two
steps. First define
\begin{equation}
\bnu_{\rm{noise,1}}(\bx) = \left(\dis\frac{1}{2}{I} -
\Kcal_\Omega^\omega \right) [\bnu_{\rm{noise}}] (\bx), \quad \bx
\in \partial \Omega. \label{eq:bnu1}
\end{equation}
Then, due to linearity,  $\bnu_{\rm{noise,1}}$ is also a mean-zero
circular Gaussian random process. Its covariance function can be
calculated as
\begin{equation*}
\begin{aligned}
& \EE [\bnu_{\rm{noise,1}}(\by) \otimes \overline{\bnu_{\rm{noise,1}}(\by')} ] = \frac{1}{4} \EE[\bnu_{\rm{noise}}(\by)\otimes \overline{\bnu_{\rm{noise}}(\by')}] - \frac{1}{2} \EE [ \Kcal_\Omega^\omega[\bnu_{\rm{noise}}](\by)\otimes \overline{\bnu_{\rm{noise}}(\by')}]\\
-\frac{1}{2} & \EE [ \bnu_{\rm{noise}}(\by)\otimes
\overline{\Kcal_\Omega^\omega[\bnu_{\rm{noise}}](\by')}] + \EE [
\Kcal_\Omega^\omega[\bnu_{\rm{noise}}](\by) \otimes
\overline{\Kcal_\Omega^\omega[\bnu_{\rm{noise}}](\by')}].
\end{aligned}
\end{equation*}
The terms on the right-hand side can be evaluated using the
statistics of $\bnu_{\rm{noise}}$ and the explicit expression of
$\Kcal_\Omega^\omega$. Let us calculate the last term. It has the
expression
\begin{equation*}
\EE \bigg[\int_{\partial \Omega} \int_{\partial \Omega}
\left[\frac{\partial \bGam_0}{\partial \nu_\bx}(\by-\bx)
\bnu_{\rm{noise}}(\bx)\right] \otimes
\overline{\left[\frac{\partial \bGam_0}{\partial
\nu_{\bx'}}(\by'-\bx') \bnu_{\rm{noise}}(\bx')\right]}
d\sigma(\bx) d\sigma(\bx') \bigg].
\end{equation*}
Using the coordinate representations and the summation convention,
we can calculate the $jk$th element of this matrix by
\begin{equation*}
\begin{aligned}
&\int_{\partial \Omega} \int_{\partial \Omega} \left[\frac{\partial \bGam_0}{\partial \nu_\bx}(\by-\bx)\right]_{jl}  \left[\frac{\partial \overline{\bGam_0}}{\partial \nu_{\bx'}}(\by'-\bx')\right]_{ks} \EE [\bnu_{\rm{noise}}(\bx) \otimes \overline{\bnu_{\rm{noise}}(\bx')}]_{ls} d\sigma(\bx) d\sigma(\bx')\\
=& \signoise^2 \int_{\partial \Omega} \left[\frac{\partial \bGam_0}{\partial \nu_\bx}(\by-\bx)\right]_{js}
\left[\frac{\partial \overline{\bGam_0}}{\partial \nu_\bx}(\by'-\bx)\right]_{ks} d\sigma(\bx)\\
=& \signoise^2 \int_{\partial \Omega} \frac{\partial
\bGam_0}{\partial \nu_\bx}(\by-\bx) \frac{\partial
\overline{\bGam_0}}{\partial \nu_\bx}(\bx-\by')d\sigma(\bx).
\end{aligned}
\end{equation*}
In the last step, we used the reciprocity relation
\begin{equation}
\bGam_0(\by-\bx) = [\bGam_0(\bx-\by)]^T, \label{eq:recip}
\end{equation}
for any $\bx, \by \in \R^d$.

The other terms in the covariance function of
$\bnu_{\rm{noise,1}}$ can be similarly calculated. Consequently,
we have
\begin{equation}
\begin{aligned}
\EE [\bnu_{\rm{noise,1}}(\by) \otimes \overline{\bnu_{\rm{noise,1}}(\by')} ] = & \frac{\signoise^2}{4} \delta_{\by}(\by') \ID - \frac{\signoise^2}{2}\left[\frac{\partial \bGam_0}{\partial \nu_{\by'}}(\by-\by') + \frac{\partial \overline{\bGam_0}}{\partial \nu_\by}(\by-\by')\right]\\
&+ \signoise^2 \int_{\partial \Omega} \frac{\partial
\bGam_0}{\partial \nu_\bx}(\by-\bx) \frac{\partial
\overline{\bGam_0}}{\partial \nu_\bx}(\bx-\by')d\sigma(\bx).
\label{eq:covnu1}
\end{aligned}
\end{equation}

From the expression of $\IWF$ and $\I_{\rm{W}}$, we see that only
the Helmholtz decomposition of $\bw_{\rm{meas}}$, that is
$\OH^P[\bw]$ and $\OH^S[\bw]$, are used in the imaging functional.
Define $\bw^\alpha = \OH^\alpha[\bw], \alpha \in \{P, S\}$. Using
the decomposition in \eqref{eq:bwdec}, we can similarly define
$\bw^\alpha_{\rm{true}}$ and $\bw^\alpha_{\rm{noise}}$. In
particular, we find that
\begin{equation*}
\bw^\alpha_{\rm{noise}}(\bx) = -\int_{\partial \Omega}
\bGam_{0,\alpha}(\bx - \by) \overline{\bnu_{\rm{noise,1}}(\by)}
d\sigma(\by), \quad \bx \in \Omega.
\end{equation*}
This is a mean zero $\CC^d$-valued circular Gaussian random field
with parameters in $\Omega$. The $jk$th element of its covariance
function is evaluated by
\begin{equation*}
\EE [ \bw^\alpha_{\rm{noise}}(\bx) \otimes \overline{\bw^\alpha_{\rm{noise}}}(\bx')]_{jk} = 
\sum_{l,s} \int_{(\partial \Omega)^2}
(\bGam_{0,\alpha}(\bx-\by))_{jl}
(\overline{\bGam_{0,\alpha}}(\bx'-\by'))_{ks}
\EE[\overline{\bnu_{\rm{noise,1}}}(\by) \otimes
\bnu_{\rm{noise,1}}(\by')]_{ls}.
\end{equation*}
Using the statistics of $\bnu_{\rm{noise,1}}$ derived above, we
find that
\begin{equation*}
\begin{aligned}
&&\EE [ \bw^\alpha_{\rm{noise}}(\bx) \otimes \overline{\bw^\alpha_{\rm{noise}}}(\bx')] =\frac{\signoise^2}{4} \int_{\partial \Omega} \bGam_{0,\alpha}(\bx-\by)\overline{\bGam_{0,\alpha}}(\by-\bx') d\sigma(\by)\\
 - \frac{\signoise^2}{2} &&\int_{(\partial \Omega)^2}\bGam_{0,\alpha}(\bx-\by)\Big[\frac{\partial \bGam_0}{\partial \nu_{\by}}(\by-\by') + \frac{\partial \overline{\bGam_0}}{\partial \nu_{\by'}}(\by-\by') \Big] \overline{\bGam_{0,\alpha}}(\by'-\bx') d\sigma(\by) d\sigma(\by')\\
+ \signoise^2 &&\int_{(\partial \Omega)^3}
\bGam_{0,\alpha}(\bx-\by)\frac{\partial
\overline{\bGam_0}}{\partial \nu_{\bz}}(\by-\bz) \frac{\partial
\bGam_0}{\partial \nu_{\bz}}(\bz-\by')
\overline{\bGam_{0,\alpha}}(\by'-\bx') d\sigma(\bz)d\sigma(\by)
d\sigma(\by').
\end{aligned}
\end{equation*}
Thanks to the Helmholtz-Kirchhoff identities, the above expression
is simplified to
\begin{eqnarray*}
&\EE [ \bw^\alpha_{\rm{noise}}(\bx) \otimes \overline{\bw^\alpha_{\rm{noise}}}(\bx')] =
&-\frac{\signoise^2}{4c_\alpha \omega} \Im m\{\bGam_{0,\alpha}(\bx-\bx')\}\notag\\
&  & + \frac{\signoise^2}{2c_\alpha \omega} \int_{\partial \Omega} \bGam_{0,\alpha}(\bx-\by) \frac{\partial \Im m\{\bGam_{0,\alpha}(\by-\bx')\}}{\partial \nu_{\by}} d\sigma(\by)\notag\\
&  & + \frac{\signoise^2}{2c_\alpha \omega} \int_{\partial \Omega} \frac{\partial \Im m\{\bGam_{0,\alpha}(\bx-\by')\}}{\partial \nu_{\by'}} \overline{\bGam_{0,\alpha}}(\by'-\bx') d\sigma(\by')\notag\\
&  & - \frac{\signoise^2}{(c_\alpha \omega)^2} \int_{\partial
\Omega} \frac{\partial \Im
m\{\bGam_{0,\alpha}(\bx-\bz)\}}{\partial \nu_{\bz}}
\frac{\partial \Im m \{\bGam_{0,\alpha}(\bz-\bx')\}}{\partial \nu_{\bz}}  d\sigma(\bz)\notag .\\
\end{eqnarray*}
Assuming that $\bx, \bx'$ are far away from the boundary, we have
from \cite{TrElastic} the asymptotic formula that
\begin{equation}
\frac{\partial \bGam_{0,\alpha}(\bx-\by)}{\partial \nu_{\by}}
\simeq ic_\alpha \omega \bGam_{0,\alpha}(\bx-\by),
\label{eq:gsommer}
\end{equation}
where the error is of order $o(|\bx-\by|^{1/2-d})$. Using this
asymptotic formula and the Helmholtz-Kirchhoff identity (taking
the imaginary part of the identity), we obtain that
\begin{equation}
\EE [ \bw^\alpha_{\rm{noise}}(\bx) \otimes
\overline{\bw^\alpha_{\rm{noise}}}(\bx')] = -
\frac{\signoise^2}{4c_\alpha \omega} \Im
m\{\bGam_{0,\alpha}(\bx-\bx')\}. \label{eq:covwalpha}
\end{equation}

In conclusion, the random field $\bw^\alpha_{\rm{noise}}(\bx), \bx
\in \Omega$, is a Gaussian field with mean zero and covariance
function (\ref{eq:covwalpha}). It is a speckle pattern, {\it
i.e.}, a random cloud of hot spots where typical diameters are of
the order of the wavelength and whose typical amplitudes are of
the order of $\signoise/(2 \sqrt{c_\alpha \omega})$.

\subsection{Stability analysis}

Now we are ready to analyze the statistical stability of the
imaging functional $\IWF$. As before, we consider separate cases
where the medium has only density contrast or only elastic
contrast.

\subsubsection{Case I: Density contrast}

Using the facts that the plane waves $\bU^P$'s are irrotational
and that the plane waves $\bU^S$'s are solenoidal, we see that for
a searching point $\bz \in \Omega$, and $\alpha \in \{ P,S\}$,
\begin{equation*}
\IWF[\{\bU_j^\alpha\}](\bz) = c_\alpha
\omega^2\left(\frac{\rho_1'}{\rho_0} - 1\right) |B'| \frac{1}{n}
\sum_{j=1}^n \Re e\{\bU^\alpha_j(\bz) \cdot
(\bw^\alpha_{j,\rm{true}}(\bz) +
\bw^\alpha_{j,\rm{noise}}(\bz))\}.
\end{equation*}
We observe the following: The contribution of
$\{\bw^\alpha_{j,\rm{true}}\}$ are exactly those in Proposition
\ref{prop-WTD-caseI}. On the other hand, the contribution of
$\{\bw^\alpha_{j,\rm{noise}}\}$ forms a field corrupting the true
image. With $C_\alpha := c_\alpha \omega^2
|B'|(\rho_1'/\rho_0-1)$, the covariance function of the corrupted
image, can be calculated as follows. Let $\bz' \in \Omega$. We
have
\begin{equation*}
\begin{aligned}
& \mathrm{Cov}
(\IWF[\{\bU_j^\alpha\}](\bz),\IWF[\{\bU_j^\alpha\}](\bz')) =
C_\alpha^2 \frac{1}{n^2} \sum_{j,l=1}^n \EE [\Re e\{\bU^\alpha_j
\cdot \bw^\alpha_{j,\rm{noise}} \}
\Re e\{\bU^\alpha_l \cdot \bw^\alpha_{l,\rm{noise}} \}]\\
=& C_\alpha^2 \frac{1}{2n^2} \sum_{j=1}^n \Re e \left\{
\bU^\alpha_j(\bz) \cdot \EE [\bw^\alpha_{j,\rm{noise}}(\bz)
\otimes \overline{\bw^\alpha_{j,\rm{noise}}}(\bz')]
\overline{\bU^\alpha_j(\bz')}\right\}.
\end{aligned}
\end{equation*}
To get the second equality, we used the fact that
$\bw^\alpha_{j,\rm{noise}}$ and $\bw^\alpha_{l,\rm{noise}}$ are
uncorrelated unless $j=l$. Thanks to the statistics
\eqref{eq:covwalpha},  the covariance of the image is given by
\begin{equation*}
- C_\alpha^2 \frac{\signoise^2}{4c_\alpha \omega}\frac{1}{2n^2}
\Re e \sum_{j=1}^n e^{i\K_\alpha(\bz-\bz')\cdot \be_{\theta_j}}
\be_{\theta_j}^\alpha \cdot [\Im m\{\bGam_{0,\alpha}(\bz-\bz')\}
\be_{\theta_j}^\alpha],
\end{equation*}
where $\be_{\theta_j}^P = \be_{\theta_j}$ and $\be_{\theta_j}^S=
\be_{\theta_j}^\perp$.

Using the same arguments as those in the proof of Proposition
\ref{prop-WTD-caseI}, we obtain that
\begin{equation}
\mathrm{Cov}
(\IWF[\{\bU_j^\alpha\}](\bz),\IWF[\{\bU_j^\alpha\}](\bz')) =
C_\alpha' \frac{\signoise^2}{2n} \lvert\Im m
\{\bGam_{0,\alpha}(\bz - \bz') \}\rvert^2, \label{eq:covspmP}
\end{equation}
where the constant $$C_\alpha' = c_\alpha\omega^3 \mu_0 |B'|^2
({\frac {\rho_1'} {\rho_0}} - 1)^2 ({\frac \pi {\K_\alpha}})^{d-2}
(\frac{\K_S}{\K_\alpha})^2 .$$

The following remarks hold. Firstly, the perturbation due to noise
has small typical values of order $\signoise/\sqrt{2n}$ and
slightly affects the peak of the imaging functional $\IWF$.
Secondly, the typical shape of the hot spot in the perturbation
due to the noise is exactly of the form of the main peak of $\IWF$
obtained in the absence of noise. Thirdly, the use of multiple
directional plane waves reduces the effect of measurement noise on
the image quality.

From (\ref{eq:covspmP}) it follows that the variance of the
imaging functional $\IWF$ at the search point $\bz$ is given by
\begin{equation} \label{vartd}
\mathrm{Var} (\IWF[\{\bU^\alpha_j\}](\bz)) = C_\alpha'
\frac{\signoise^2}{2n} \lvert\Im m \{\bGam_{0,\alpha}(\bm{0})
\}\rvert^2.
\end{equation}
Define the  Signal-to-Noise Ratio (SNR) by
$$
\mathrm{SNR} :=
\frac{\EE[\IWF[\{\bU^\alpha_j\}](\bz_a)]}{\mathrm{Var}
(\IWF[\{\bU^\alpha_j\}](\bz_a))^{1/2}},
$$
where $\bz_a$ is the true location of the inclusion. From
(\ref{eqgf1}),  (\ref{eqgf2}), and (\ref{vartd}), we have
\begin{equation}
\mathrm{SNR} = \frac{4  \sqrt{2 \pi^{d-2} n \omega^{5-d}
\rho_0^{3} c_\alpha^{d-1}} \delta^d |B|
|\rho_1-\rho_0|}{\signoise} \lvert\Im m \{\bGam_{0,\alpha}(\bm{0})
\}\rvert. \label{eq:SNRm}
\end{equation}
From (\ref{eq:SNRm}), the SNR is proportional to the contrast
$|\rho_1-\rho_0|$ and the volume of the inclusion $\delta^d |B|$,
over the standard deviation of the noise, $\signoise$.
\subsubsection{Case II: Elasticity contrast}

In the case of elastic contrast, the  imaging functional becomes
for $\bz\in \Omega$
\begin{equation}
\IWF[\{\bU_j^\alpha\}](\bz) = c_\alpha \frac{1}{n} \sum_{j=1}^n
\nabla  \bU^\alpha_j(\bz) : \MM'(B') (\nabla
\bw^\alpha_{j,\rm{true}}(\bz) + \nabla
\bw^\alpha_{j,\rm{noise}}(\bz)).
\label{eq:IWF_e}
\end{equation}
Here, $\bw^\alpha_{j,\rm{true}}$ and $\bw^\alpha_{j,\rm{noise}}$
are defined in the last section. They correspond to the
backpropagation of pure data and that of the measurement noise.
The contribution of $\bw^\alpha_{j,\rm{true}}$ is exactly the
imaging functional with unperturbed data and it is investigated in
Proposition \ref{prop-WTD-caseII}. The contribution of
$\bw^\alpha_{j,\rm{noise}}$ perturbs the true image. For $\bz,
\bz' \in \Omega$,  the covariance function of the TD noisy image
is given by
\begin{equation*}
\begin{aligned}
& \mathrm{Cov}
(\IWF[\{\bU_j^\alpha\}](\bz),\IWF[\{\bU_j^\alpha\}](\bz'))\\=
&c_\alpha^2 \frac{1}{n^2} \sum_{j,l=1}^n \EE [\Re
e\{\nabla\bU^\alpha_j(\bz) : \MM' \nabla
\bw^\alpha_{j,\rm{noise}}(\bz)\}
\Re e\{\nabla \bU^\alpha_l(\bz') : \MM' \nabla \bw^\alpha_{l,\rm{noise}}(\bz') \}]\\
= &c_\alpha^2 \frac{1}{2 n^2} \sum_{j,l=1}^n \Re e\EE
[(\nabla\bU^\alpha_j(\bz) : \MM' \nabla
\bw^\alpha_{j,\rm{noise}}(\bz)) \overline{(\nabla
\bU^\alpha_l(\bz') : \MM' \nabla
 \bw^\alpha_{l,\rm{noise}}(\bz'))}]\\
=& c_\alpha^2 \frac{1}{2n^2} \sum_{j=1}^n \Re e \left\{ \nabla
\bU^\alpha_j(\bz) : \MM' \Big[\EE [\nabla
\bw^\alpha_{j,\rm{noise}}(\bz) \overline{\nabla
\bw^\alpha_{j,\rm{noise}}}(\bz')]:\MM' \overline{\nabla
\bU^\alpha_j}(\bz') \Big]\right\}.
\end{aligned}
\end{equation*}
Using \eqref{eq:covwalpha}, we find that
\begin{equation*}
\EE [\nabla \bw^\alpha_{j,\rm{noise}}(\bz) \overline{\nabla
\bw^\alpha_{j,\rm{noise}}}(\bz')] = - \frac{\signoise^2}{4c_\alpha
\omega} \Im m \nabla_{\bz} \nabla_{\bz'}
\{\bGam_{0,\alpha}(\bz-\bz')\}.
\end{equation*}
After substituting this term into the expression of the covariance
function, we find that it becomes
\begin{equation*}
\frac{c_\alpha \signoise^2}{4\omega} \frac{1}{2n} \sum_{j=1}^n \Re
e \left\{ \nabla \bU^\alpha_j(\bz) : \MM' \Big[\Im m\big\{
\nabla^2 \bGam_{0,\alpha}(\bz-\bz') \big\}:\MM' \overline{\nabla
\bU^\alpha_j}(\bz') \Big]\right\}.
\end{equation*}
The sum has exactly the form that was analyzed in the proof of
Proposition \ref{prop-TD-caseII}. Using similar techniques, we
finally obtain that
\begin{equation}
\mathrm{Cov}
(\IWF[\{\bU_j^\alpha\}](\bz),\IWF[\{\bU_j^\alpha\}](\bz')) = \mu_0
\big(\frac{c_\alpha}{\omega}\big)^3\big(\frac{\pi}{\K_\alpha}\big)^{d-2}
(\frac{\K_S}{\K_\alpha})^2 \frac{\signoise^2}{2n}
J_{\alpha,\alpha}(\bz,\bz'), \label{eq:covIWne}
\end{equation}
where $J_{\alpha,\alpha}$ is defined by \eqref{J}.  The variance
of the TD image can also be obtained from \eqref{eq:covIWne}. As
in the case of density contrast, the typical shape of hot spots in
the image corrupted by noise is the same as the main peak of the
true image. Further, the effect of measurement noise is reduced by
a factor of $\sqrt{n}$ by using $n$ plane waves. In particular,
the SNR of the TD image is given by
\begin{equation}
\mathrm{SNR} = \frac{\delta^d \sqrt{\mu_0
\omega}}{\sqrt{c_\alpha^3}}
\big(\frac{\pi}{\K_\alpha}\big)^{\frac{d-2}{2}}
\frac{\K_S}{\K_\alpha} \frac{4\sqrt{2n}}{\signoise}
\sqrt{J_{\alpha,\alpha}(\bz_a,\bz_a)}.
\end{equation}
\section{Statistical stability with medium noise}
\label{sec:ssmedium}

In the previous section, we demonstrated that the proposed imaging
functional using multi-directional plane waves is statistically
stable with respect to uncorrelated measurement noises. Now we
investigate the case of medium noise, where the constitutional
parameters of the elastic medium fluctuate around a constant
background.

\subsection{Medium noise model}

For simplicity, we consider a medium that fluctuates in the
density parameter only. That is,
\begin{equation}
\rho(\bx) = \rho_0 [1+\bmu(\bx)], \label{eq:rhonoi}
\end{equation}
where $\rho_0$ is the constant background and $\rho_0\bmu(\bx)$ is
the random fluctuation in the density. Note that $\bmu$ is real
valued.

Throughout this section, we will call the homogeneous medium with
parameters $(\lambda_0,\mu_0,\rho_0)$ the reference medium. The
background medium refers to the one without inclusion but with
density fluctuation. Consequently, the background Neumann problem
of elastic waves is no longer \eqref{Background-Sol}. Indeed, that
equation corresponds to the reference medium and its solution will
be denoted by $\bU^{(0)}$. The new background solution is
\begin{equation}\label{eq:rbgsolu}
\left\{
\begin{array}{ll}
(\OL_{\lambda_0,\mu_0}+\rho_0\omega^2[1+\bmu])\bU =0, & \text{on
}\Omega,
\\\nm
\dis\frac{\partial\bU}{\partial\nu}=\bg & \text{on
}\partial\Omega,
\end{array}
\right.
\end{equation}
Similarly, the Neumann function associated to the problem in the
reference medium will be denoted by $\NN^{\omega,(0)}$. We denote
by $\NN^\omega$ the Neumann function associated to the background
medium, that is,
\begin{equation}\label{eq:rNeumEq}
\left\{
\begin{array}{ll}
(\OL_{\lambda_0,\mu_0}+\rho_0\omega^2[1+\bmu(\bx)])\NN^\omega
(\bx,\by) = - \delta_\by(\bx)\ID, &  \bx\in\Omega,\quad \bx\neq
\by,
\\\nm
\dis\frac{\partial\NN^\omega }{\partial\nu}(\bx,\by)=0 &
\bx\in\partial\Omega.
\end{array}
\right.
\end{equation}

We assume that $\bmu$ has small amplitude so that the Born
approximation is valid. In particular, we have
\begin{equation}
\NN^\omega(\bx,\by) \simeq \NN^{\omega,(0)}(\bx,\by) + \rho_0
\omega^2 \int_\Omega
\NN^{\omega,(0)}(\bx,\bz)\bmu(\bz)\NN^{\omega,(0)}(\bz,\by) d\by.
\label{eq:NBa}
\end{equation}
As a consequence, we also have that $\bU \simeq \bU^{(0)} -
\bU^{(1)}$ where
\begin{equation}
\bU^{(1)}(\bx) = - \rho_0 \omega^2\int_\Omega
\NN^{\omega,(0)}(\bx, \bz) \bmu(\bz) \bU^{(0)}(\bz) d\bz.
\label{eq:UBa}
\end{equation}
Let $\sigmu$ denotes the typical size of $\bmu$, the remainders in
the above approximations are of order $o(\sigmu)$.

\subsection{Statistics of the speckle field in the case of a density contrast
only} \label{sec:mednoise}

We assume that the inclusion has density contrast only. The
backpropagation step constructs $\bw$ as follows:
\begin{equation}
\bw(\bx) = \int_{\partial \Omega}
\bGam_0(\bx,\bz)\overline{(\frac{1}{2}I -
\Kcal^{\omega,(0)}_\Omega)[\bU^{(0)}-\bu_{\rm{meas}}]}(\bz)
d\sigma(\bz),\quad  \bx \in \Omega. \label{eq:wredef}
\end{equation}
We emphasize that the backpropagation step uses the reference
fundamental solutions, and the differential measurement is with
respect to the reference solution. These are necessary steps
because of the fluctuation in the background medium or
equivalently, because of the fact that the background solution is
unknown.

Writing the difference between $\bU^{(0)}$ and $\bu_{\rm{meas}}$
as the sum of $\bU^{(0)} - \bU$ and $\bU - \bu_{\rm{meas}}$. These
two differences are estimated by $\bU^{(1)}$ in \eqref{eq:UBa} and
by \eqref{Asymptotic-Exp}, respectively. Using Lemma
\ref{N-K-Gamma}, we find that
\begin{equation}\label{eq:wmed}
\begin{aligned}
\bw(\bx) = &\ - \rho_0\omega^2 \int_{\partial \Omega} \bGam_0 (\bx - \bz)\int_{\Omega} \overline{\bGam_0}(\bz-\by)\overline{\bU^{(0)}}(\by) \bmu(\by) d\by d\sigma(\bz)\\
\ & - C\delta^d \int_{\partial \Omega}
\bGam_0(\bx-\bz)\overline{\bGam_0}(\bz-\bz_a)
\overline{\bU^{(0)}}(\bz_a) d\sigma(\bz) + O(\sigmu \delta^d) +
o(\sigmu), \quad \bx \in \Omega,
\end{aligned}
\end{equation}
where $C = \omega^2(\rho_0-\rho_1)|B|$. The second term is the
leading contribution of $\bU-\bu_{\rm{meas}}$ given by
approximating the unknown Neumann function and the background
solution by those associated to the reference medium. The leading
error in this approximation is of order $O(\sigmu \delta^d)$ and
can be written explicitly as
\begin{equation*}
\begin{aligned}
C\rho_0 \omega^2 \delta^d &\int_{\partial \Omega} \bGam_0(\bx,\bz)  \int_{\Omega} \overline{\bGam_0}(\bz,\by) \overline{\NN^{\omega,(0)}}(\by,\bz_a) \overline{\bU^{(0)}}(\bz_a) \bmu(\by) d\by d\sigma(\bz)\\
& - C\rho_0 \omega^2  \delta^d  \int_{\partial \Omega}
\bGam_0(\bx,\bz) \overline{\bGam_0}(\bz,\bz_a) \int_{\Omega}
\overline{\NN^{\omega,(0)}}(\bz_a,\by) \overline{\bU^{(0)}}(\by)
\bmu(\by) d\by d\sigma(\bz),
\end{aligned}
\end{equation*}
and is neglected in the sequel.

For the Helmholtz decomposition $\bw^\alpha, \alpha \in \{P,S\}$,
the first fundamental solution $\bGam_0(\bx-\bz)$ in the
expression \eqref{eq:wmed} should be changed to
$\bGam_{0,\alpha}(\bx-\bz)$. We observe that the second term in
\eqref{eq:wmed} is exactly \eqref{39}. Therefore, we call this
term $\bw_{\rm{true}}$ and refer to the other term in the
expression as $\bw_{\rm{noise}}$. Using the Helmholtz-Kirchhoff
identity, we obtain
\begin{equation}
\label{eq:bwnoise} \bw^\alpha_{\rm{noise}}(\bx) \simeq -
\frac{\rho_0 \omega}{c_\alpha} \int_{\Omega} \bmu(\by) \Im
m\{\bGam_{0,\alpha}(\bx-\by)\} \overline{\bU^{(0)}}(\by) d\by,
\quad \bx \in \Omega.
\end{equation}


We have decomposed the backpropagation $\bw^\alpha$ into the
``true'' $\bw^\alpha_{\rm{true}}$ which behaves like in reference
medium and the error part $\bw^\alpha_{\rm{noise}}$. In the TD
imaging functional using multiple plane waves with
equi-distributed directions, the contribution of
$\bw^\alpha_{\rm{true}}$ is exactly as the one analyzed in
Proposition \ref{prop-WTD-caseI}. The contribution of
$\bw^\alpha_{\rm{noise}}$ is a speckle field.

The covariance function of this speckle field, or equivalently
that of the TD image corrupted by noise, is
\begin{equation*}
\mathrm{Cov}
(\IWF[\{\bU_j^\alpha\}](\bz),\IWF[\{\bU_j^\alpha\}](\bz')) =
C_\alpha^2 \frac{1}{n^2} \sum_{j,l=1}^n \EE [\Re
e\{\bU^{(0),\alpha}_j(\bz) \cdot \bw^\alpha_{j,\rm{noise}}(\bz) \}
\Re e\{\bU^{(0),\alpha}_l (\bz)\cdot
\bw^\alpha_{l,\rm{noise}}(\bz) \}],
\end{equation*}
for $\bz, \bz' \in \Omega$, where $C_\alpha$ is defined to be
$c_\alpha \omega^2|B'|(\rho'_1/\rho_0-1)$. Here $\bU^{(0),\alpha}$
are the reference incoming plane waves (\ref{plane-waves}).

Using the expression \eqref{eq:bwnoise}, we have
\begin{equation*}
\begin{aligned}
\frac{1}{n} \sum_{j=1}^n \bU^\alpha_j(\bz) \cdot
\bw^\alpha_{j,\rm{noise}}(\bz) &= - b_\alpha \frac{1}{n}
\sum_{j=1}^n\int_\Omega \bmu(\by)
\big[\bU^{(0),\alpha}_j(\bz)\otimes
\overline{\bU^{(0),\alpha}_j}(\by)\big] : \Im m\{\bGam_{0,\alpha}(\bz-\by)\} d\by\\
&= - b_\alpha \int_\Omega \bmu(\by) \frac{1}{n} \sum_{j=1}^n
e^{i\K_\alpha(\bz-\by)\cdot \be_{\theta_j}}  \be_{\theta_j}^\alpha
\otimes \be_{\theta_j}^\alpha : \Im m\{\bGam_{0,\alpha}(\bz-\by)\}
d\by.
\end{aligned}
\end{equation*}
where $b_\alpha = (\rho_0 \omega)/c_\alpha$. Finally, using
\eqref{approx-e-times-e} and \eqref{approx-eP-times-eP} for
$\alpha = P$ and $S$ respectively, we obtain that
\begin{equation}
\frac{1}{n} \sum_{j=1}^n \bU^{(0),\alpha}_j(\bz) \cdot
\bw^\alpha_{j,\rm{noise}}(\bz) = b'_\alpha \int_\Omega \bmu(\by)
\lvert \Im m\{\bGam_{0,\alpha}(\bz-\by)\} \rvert^2 d\by.
\label{eq:bUtimesbw}
\end{equation}
Here $b'_\alpha = 4b_\alpha\mu_0(\frac{\pi}{\K_\alpha})^{d-2}
(\frac{\K_S}{\K_\alpha})^2$. Note that the sum above is a real
quantity.

The covariance function of the TD image simplifies to
\begin{equation}
C_\alpha^2 {b'_\alpha}^2 \int_\Omega \int_\Omega C_\bmu(\by,\by')
\lvert \Im m\{\bGam_{0,\alpha}(\bz-\by)\}\rvert^2 \lvert \Im
m\{\bGam_{0,\alpha}(\bz'-\by')\}\rvert^2 d\by d\by',
\label{eq:covIFmed}
\end{equation}
where $C_\bmu(\by,\by') = \EE[\bmu(\by) \bmu(\by')]$ is the
two-point correlation function of the fluctuations in the density
parameter.

\begin{rem}\label{rem:corrlen}
The expression in \eqref{eq:bUtimesbw} shows that the speckle
field in the image is essentially the medium noise smoothed by an
integral kernel of the form $\lvert\Im m
\bGam_{0,\alpha}\rvert^2$. Similarly, \eqref{eq:covIFmed} shows
that the correlation structure of the speckle field is essentially
that of the medium noise smoothed by the same kernel. Because the
typical width of this kernel is about half the wavelength, the
correlation length of the speckle field is roughly the maximum
between the correlation length of medium noise and the wavelength.
\end{rem}


\subsection{Statistics of the speckle field in the case of an elasticity contrast}

The case of elasticity contrast can be considered similarly. The
covariance function of the TD image is
\begin{equation*}
c_\alpha^2 \frac{1}{n^2} \sum_{j,l=1}^n \EE [\Re
e\{\nabla\bU^{(0),\alpha}_j(\bz) : \MM' \nabla
\bw^\alpha_{j,\rm{noise}}(\bz)\} \Re e\{\nabla
\bU^{(0),\alpha}_l(\bz') : \MM' \nabla
\bw^\alpha_{l,\rm{noise}}(\bz') \}].
\end{equation*}
Using the expression of $\bw^\alpha_{\rm{noise}}$, we have
\begin{equation*}
\begin{aligned}
\frac{1}{n} \sum_{j=1}^n \nabla\bU^\alpha_j(\bz) : \MM' \nabla
\bw^\alpha_{j,\rm{noise}}(\bz) = -
 b_\alpha & \int_\Omega \bmu(\by) \frac{1}{n} \sum_{j=1}^n i\K_\alpha e^{i\K_\alpha(\bz-\by)\cdot \be_{\theta_j}} \\
&\be_{\theta_j}\otimes \be_{\theta_j}^\alpha \otimes
\be_{\theta_j}^\alpha : \big[\MM' \Im m\{\nabla_\bz
\bGam_{0,\alpha}(\bz-\by)\} \big] d\by.
\end{aligned}
\end{equation*}
From \eqref{approx-e-times-e} and \eqref{approx-eP-times-eP}, we
see that
\begin{equation}
\frac{1}{n} \sum_{j=1}^n i\K_\alpha e^{i\K_\alpha \bx \cdot
\be_{\theta_j}} \be_{\theta_j}\otimes \be_{\theta_j}^\alpha
\otimes \be_{\theta_j}^\alpha = - 4\mu_0
\big(\frac{\pi}{\K_\alpha}\big)^{d-2} (\frac{\K_S}{\K_\alpha})^2
\Im m\{\nabla \bGam_{0,\alpha}(\bx)\}. \label{eq:sumeee}
\end{equation}
Using this formula, we get
\begin{equation}
\frac{1}{n} \sum_{j=1}^n \nabla\bU^\alpha_j(\bz) : \MM' \nabla
\bw^\alpha_{j,\rm{noise}}(\bz) = b_\alpha' \int_\Omega \bmu(\by)
Q^2_\alpha[\MM'] (\bz-\by) d\by, \label{eq:sumubwE}
\end{equation}

where $Q^2_\alpha[\MM'](\bx)$ is a non-negative function defined
as
\begin{equation}
\begin{aligned}
Q^2_P[\MM](\bx) &= \Im m\{\nabla \bGam_{0,P}(\bx) \} :
[\MM \Im m\{\nabla \bGam_{0,P} (\bx)\}]\\
&= a\lvert \Im m \{\nabla \bGam_{0,P}(\bx)\} \rvert^2 + b\lvert
\Im m\{\nabla \cdot \bGam_{0,P}(\bx)\}\rvert^2. \label{eq:Qdef}
\end{aligned}
\end{equation}
The last equality follows from the expression \eqref{M-disk2} of
$\MM$ and the fact that $\partial_i (\bGam_{0,P})_{jk} =
\partial_j (\bGam_{0,P})_{ik}$. This symmetry is not satisfied for $\bGam_{0,P}$ for which we have
\begin{equation}
\begin{aligned}
& Q^2_S[\MM](\bx) = \Im m\{\nabla \bGam_{0,S}(\bx) \} :
[\MM \Im m\{\nabla \bGam_{0,S} (\bx)\}]\\
= \ &\frac{a}{2} \lvert \Im m\{\nabla \bGam_{0,S}(\bx)\} \rvert^2
+ \frac{a}{2} \Im m\{\nabla \bGam_{0,S}(\bx)\}: \Im
m\{\widetilde{\nabla} \bGam_{0,S}(\bx)\}+ b\lvert \Im m\{\nabla
\cdot \bGam_{0,S}(\bx)\} \rvert^2. \label{eq:QdefS}
\end{aligned}
\end{equation}
Here $(\nabla \cdot \bGam_{0,S}(\bx))_{jkl} = \partial_k
(\bGam_{0,S}(\bx))_{jl}$. Note that $Q^2_\alpha$ is non-negative
and \eqref{eq:sumubwE} is real valued.

The covariance function of the TD image  simplifies to
\begin{equation}
c_\alpha^2 {b'_\alpha}^2 \int_\Omega \int_\Omega C_\bmu(\by,\by')
Q^2_\alpha[\MM'](\bz-\by) Q^2_\alpha[\MM] (\bz'-\by') d\by d\by',
\quad \bz, \bz'\in \Omega . \label{eq:covIFmedE}
\end{equation}

\begin{rem} If we compare \eqref{eq:sumubwE} with \eqref{eq:bUtimesbw}, then one can see that they are of the same form except that the integral
kernel is now $Q^2_\alpha[\MM']$. Therefore, Remark
\ref{rem:corrlen} applies here as well. We remark also that the
further reduction of the effect of measurement noise with rate
$1/\sqrt{2n}$ does not appear in the medium noise case. In this
sense, TD imaging is less stable with respect to medium noise.
\end{rem}

\subsection{Random elastic medium}

In this section we consider the case when the random fluctuation
occurs in the elastic coefficients. This is a more delicate case
because it is well known that inhomogeneity in the Lam\'e
coefficients cause mode conversion. Nevertheless, we demonstrate
below that as long as the random fluctuation is weak so that the
Born approximation is valid, the imaging functional we proposed
remains stable.

To simplify the presentation, we assume that random fluctuation
occurs only in the shear modulus $\mu$ while the density $\rho_0$
and the first Lam\'e coefficient $\lambda_0$ are homogeneous. The
equation for time-harmonic elastic wave is then
\begin{equation}
\rho_0 \omega^2 \Bu + \lambda_0 \nabla (\nabla \cdot \Bu) + \nabla
\cdot [\mu(\bx)(\nabla \Bu + (\nabla \Bu)^T)] = 0,
\label{eq:rmuelastic}
\end{equation}
with the same boundary condition as before. The inhomogeneous
shear modulus is given by
\begin{equation}
\mu(\bx) = \mu_0 + \gamma(\bx,\omega), \label{eq:randmu}
\end{equation}
where $\gamma(\bx)$ is a random process modeling the fluctuation.

{\bfseries Born approximation.} The equation for elastic wave
above can be written as
\begin{equation*}
\OL_{\lambda_0,\mu_0} \Bu + \rho_0 \omega^2 \Bu = - \nabla \cdot
[\gamma(\bx)(\nabla \Bu + (\nabla \Bu )^T)].
\end{equation*}
Assume that the random fluctuation $\gamma$ is small enough so
that the Born approximation is valid. We then have $\Bu \simeq
\Bu_0 - \Bu_1$ where $\Bu_0$ solves the equation in the background
medium and $\Bu_1$, the first scattering, solves the above
equation with $\Bu$ on the right hand side replaced by $\Bu_0$.
More precisely, we have
\begin{equation}
\Bu_1(\bx) = \int_\Omega \BN^\omega(\bx,\by){\nabla \cdot
[\gamma(\by)(\nabla \Bu + (\nabla \Bu )^T)]} d\by.
\label{eq:muBorn}
\end{equation}
Here, $\BN^\omega$ is the Neumann function in the background
medium without random fluctuation.

{\bfseries Post-processing step.} As seen before, the
post-processing \eqref{W-Def} is a critical step in our method. As
discussed in section \ref{sec:mednoise}, even when the medium is
random we have to use the reference Green's function and the
reference solution associated to the homogeneous medium in this
post-processing step. Following the analysis in section
\ref{sec:mednoise}, we see that as in \eqref{eq:wmed} the function
$\Bw$ contains two main contributions: Firstly, back-propagating
the difference between the measurement and the background solution
in the random medium but without inclusion contributes to the
detection of inclusion. Secondly, back-propagating the difference
between the background solution and the reference solution in the
homogeneous medium amounts to a speckle pattern in the image.

The first contribution corresponds to the case with exact data and
is discussed in section 4. We focus on the second contribution
which accounts for the statistical stability. This part of the
post-processed function $\Bw$ has the expression
\begin{equation*}
\begin{aligned}
\Bw^\alpha_{\mathrm{noise}}(\bz) = &\OH^\alpha [\S^\omega_\Omega
\overline{(\frac{1}{2}I - \Kcal^\omega_\Omega)\Bu_1}]
= \OH^\alpha [\int_{\partial \Omega} \bGam_0(\bz,\by) \overline{[(\frac{1}{2}I -
\Kcal^\omega_\Omega)\int_\Omega \bN^\omega(\cdot,\bx) \Bv(\bx)d\bx]}(\by) d\sigma(\by)]\\
= &\int_{\partial \Omega} \bGam_{0,\alpha}(\bz,\by) \int_\Omega
\overline{\bGam_0}(\by,\bx) \nabla \cdot[\gamma(\bx)
\overline{(\nabla \Bu + (\nabla \Bu )^T)}(\bx)]d\bx d\sigma(\by).
\end{aligned}
\end{equation*}
In the second equality above, $\Bv(\bx)$ is a short-hand notation
for the divergence term in the line below. We refer to this term
as the first scattering source. Using the Helmholtz-Kirchhoff
identity again, we obtain that
\begin{equation}
\Bw_{\alpha,\mathrm{noise}}(\bz) = \frac{1}{c_\alpha \omega}
\int_\Omega \Im m\{\bGam_{0,\alpha}(\bz,\bx)\} \nabla
\cdot[\gamma(\bx) \overline{(\nabla \Bu + (\nabla \Bu
)^T)}(\bx)]d\bx \label{eq:delmc}
\end{equation}

\begin{rem}\label{rem:hdec} Compare the above expression with that in \eqref{eq:bwnoise}. The first scattering source in \eqref{eq:bwnoise} is exactly the incident wave in the case with density fluctuation but is more complicated in the case with elastic fluctuation, see \eqref{eq:fsmuP} below. This shows that the Born approximation in an inhomogeneous medium indeed captures weak mode conversion. Nevertheless, \eqref{eq:delmc} shows that our method, due to the Helmholtz-Kirchhoff identity and our proposal of using Helmholtz decomposition, extracts only the modes that are desired by the imaging functional. As we will see, this is crucial to the statistical stability of the imaging functional.
\end{rem}

{\bfseries The speckle field.} The specific expression of the imaging function depends on the type of the inclusion and the type of probing planewaves. We first consider the case of a density inclusion. For a pressure wave $\bU^P =
e^{i\K_P \bx \cdot \be_{\theta}}$, the first scattering source is
\begin{equation}
\Bv(\bx) = 2i\K_P\nabla \cdot(\gamma(\bx)e^{i\K_P \bx \cdot
\be_{\theta}} \be_{\theta} \otimes \be_{\theta}) = 2i\K_P (\nabla
\gamma \cdot \be_{\theta}) \bU^P(\bx) - 2\K_P^2 \gamma(\bx)
\bU^P(\bx). \label{eq:fsmuP}
\end{equation}
The speckle field in the imaging functional with a set of pressure
waves $\{\bU^P_j\}$ is
\begin{equation*}
\begin{aligned}
&\I_{\mathrm{WF,noise}}[\{\bU^P_j\}](\bz) =  c_P\omega^2(\frac{\rho'_1}{\rho_0}-1)|B'| \frac{1}{n} \Re e \sum_{j=1}^n \bU_j^P(\bz) \cdot \Bw^P_{\mathrm{noise}}(\bz)\\
=\ &\omega(\frac{\rho'_1}{\rho_0}-1)|B'| \Re e \int_\Omega -2\K_P^2 \gamma(\bx) \frac{1}{n} \sum_{j=1}^n e^{i\K_P (\bz - \bx)\cdot \be_{\theta_j}} \be_{\theta_j} \otimes \be_{\theta_j} : \Im m \{\bGam_{0,P}(\bz,\bx)\}\\
& \quad - 2\frac{1}{n} \sum_{j=1}^n i\K_P e^{i\K_P(\bz-\bx)\cdot
\be_{\theta_j}} \be_{\theta_j} \otimes \be_{\theta_j} \otimes
\be_{\theta_j} : [\nabla \gamma(\bx) \otimes \Im m
\{\bGam_{0,P}(\bz,\bx)\}] d\bx.
\end{aligned}
\end{equation*}
Using the summation formulas \eqref{approx-e-times-e} and
\eqref{eq:sumeee}, we can rewrite the above quantity as
\begin{equation*}
C^P_1 \int_\Omega 2\K_P^2 \gamma(\bx) \lvert \Im
m\{\bGam_{0,P}(\bz,\bx)\} \rvert^2 + 2\Im m\{\nabla_\bz
\bGam_{0,P}(\bz,\bx)\} : [\nabla \gamma(\bx) \otimes\Im
m\{\bGam_{0,P}(\bz,\bx)\}] d\bx.
\end{equation*}
Here, the constant is
\begin{equation*}
C^P_1 = 4\mu_0 (\frac{\pi}{\K_P})^{d-2} (\frac{\K_S}{\K_P})^2
\omega(\frac{\rho'_1}{\rho_0} - 1)|B'| = 4\pi^{d-2}
\frac{\omega^3}{\K_P^d}(\rho'_1-\rho_0)|B'|.
\end{equation*}
Assuming that $\gamma=0$ near the boundary and using the
divergence theorem, we can further simplify the expression of the
speckle field to
\begin{equation*}
C_1^P \int_\Omega \gamma(\bx)[(2\K_P^2 I + \Delta_\bx)\lvert \Im
m\{\bGam_{0,P}(\bz,\bx)\}\rvert^2] d\bx = C_1^P \int_\Omega
[(2\K_P^2 I + \Delta)\gamma(\bx)] \lvert \Im
m\{\bGam_{0,P}(\bz,\bx)\}\rvert^2 d\bx.
\end{equation*}
This expression is again of the form of \eqref{eq:bUtimesbw} and
\eqref{eq:sumubwE} except that the integral kernel is more
complicated. Its correlation can be similarly calculated.
Furthermore, Remarks \ref{rem:corrlen} and \ref{rem:hdec} apply
here. The salient feature of
the speckle field does not change. It is essentially the medium
noise (or the derivative of the medium noise) smoothed by an
integral kernel whose width is of the order of the wavelength. The
correlation length of the speckle field is of the order of the
maximum between that of the medium noise and the wavelength.

The case when shear waves are used in \eqref{eq:IWFdef} can be similarly considered. For a shear wave  $\bU^S= e^{i\K_S \bx\cdot \be_{\theta}} \be_{\theta}^\perp$, the first scattering source becomes
\begin{equation}
\begin{aligned}
\Bv(\bx) &= i\K_S\nabla \cdot(\gamma(\bx)e^{i\K_S \bx \cdot \be_{\theta}} [\be_{\theta} \otimes \be_{\theta}^\perp + \be_{\theta}^\perp \otimes \be_{\theta}])\\
&= - \K_S^2 \gamma(\bx) \bU^S(\bx) + i\K_S (\nabla \gamma \cdot \be_{\theta}) \bU^S(\bx) + i\K_S (\nabla \gamma \cdot \be_{\theta}^\perp) e^{i\K_S \bx \cdot \be_{\theta}} \be_{\theta}.
\end{aligned}
\label{eq:fsmuS}
\end{equation}
The speckle field in the imaging functional with a set of shear
waves $\{\bU^S_j\}$ is
\begin{equation*}
\begin{aligned}
&\I_{\mathrm{WF,noise}}[\{\bU^S_j\}](\bz) =  c_S\omega^2(\frac{\rho'_1}{\rho_0}-1)|B'| \frac{1}{n} \Re e \sum_{j=1}^n \bU_j^S(\bz) \cdot \Bw^S_{\mathrm{noise}}(\bz)\\
=\ &\omega(\frac{\rho'_1}{\rho_0}-1)|B'| \Re e \int_\Omega -\K_S^2 \gamma(\bx) \frac{1}{n} \sum_{j=1}^n e^{i\K_S (\bz - \bx)\cdot \be_{\theta_j}} \be_{\theta_j}^\perp \otimes \be_{\theta_j}^\perp : \Im m \{\bGam_{0,S}(\bz,\bx)\}\\
& \quad - \frac{1}{n} \sum_{j=1}^n i\K_S e^{i\K_S(\bz-\bx)\cdot \be_{\theta_j}} \be_{\theta_j} \otimes \be_{\theta_j}^\perp \otimes \be_{\theta_j}^\perp : [\nabla \gamma(\bx) \otimes\Im m \{\bGam_{0,S}(\bz,\bx)\}] \\
& \quad - \frac{1}{n} \sum_{j=1}^n i\K_S e^{i\K_S(\bz-\bx)\cdot \be_{\theta_j}} \be_{\theta_j}^\perp \otimes \be_{\theta_j}^\perp \otimes \be_{\theta_j} : [\nabla \gamma(\bx) \otimes \Im m \{\bGam_{0,S}(\bz,\bx)\}] d\bx\\
\end{aligned}
\end{equation*}
Using the summation formulas \eqref{approx-eP-times-eP} and \eqref{eq:sumeee}, we can rewrite the above quantity as 
\begin{equation*}
\begin{aligned}
C^S_1 \int_\Omega \K_S^2 &\gamma(\bx) \lvert \Im m\{\bGam_{0,S}(\bz,\bx)\} \rvert^2 + \Im m\{\nabla_{\bz} \bGam_{0,S}(\bz,\bx)\} : [\nabla \gamma(\bx) \otimes \Im m\{\bGam_{0,S}(\bz,\bx)\}]\\
& + \Im m\{(\nabla_{\bz} \bGam_{0,S})^T (\bz,\bx)\} : [\nabla \gamma(\bx) \otimes \Im m\{\bGam_{0,S}(\bz,\bx)\}]d\bx.
\end{aligned}
\end{equation*}
Here, the constant is
\begin{equation*}
C^S_1 = 4\mu_0 (\frac{\pi}{\K_S})^{d-2} \omega(\frac{\rho'_1}{\rho_0} - 1)|B'| = 4\pi^{d-2} \frac{\omega^3}{\K_S^d}(\rho'_1-\rho_0)|B'|.
\end{equation*} 
Moreover, the notation $(\nabla\bGam_{0,S})^T$ means the following:
\begin{equation*}
(\nabla\bGam_{0,S}(\bx))^T_{jkl} = \partial_l (\bGam_{0,S}(\bx))_{jk}.
\end{equation*}
Again, the speckle field is a smoothed version of the medium noise and its derivatives. The smoothing kernel can be read from the expression above.

\bigskip
{\bfseries Elastic inclusion with shear waves.}  Now we consider the case of an elastic inclusion. The imaging function takes the form of \eqref{eq:IWF_e}. When a set of pressure waves $\{\bU^P_j\}$ are used, the speckle field in the imaging function can be calculated as follows.
\begin{equation*}
\begin{aligned}
&\I_{\mathrm{WF,noise}}[\{\bU^P_j\}](\bz) =  c_P\frac{1}{n} \Re e \sum_{j=1}^n \nabla \bU_j^P(\bz) : \MM'(B') \nabla \Bw^P_{\mathrm{noise}}(\bz)\\
=\ &\frac{1}{\omega}\Re e \int_\Omega -\gamma(\bx) \frac{2\K_P^2 }{n} \sum_{j=1}^n i\K_P e^{i\K_P (\bz - \bx)\cdot \be_{\theta_j}} \be_{\theta_j} \otimes \be_{\theta_j} \otimes \be_{\theta_j}: \MM'(B')\Im m \{\nabla_\bz \bGam_{0,P}(\bz,\bx)\}\\
&-\frac{2}{n} \sum_{j=1}^n (i\K_P)^2 e^{i\K_P(\bz-\bx)\cdot \be_{\theta_j}} \be_{\theta_j} \otimes \be_{\theta_j} \otimes \be_{\theta_j} \otimes \be_{\theta_j} : [\MM'(B')\Im m \{\nabla_\bz \bGam_{0,P}(\bz,\bx)\} \otimes \nabla \gamma(\bx)] d\bx.
\end{aligned}
\end{equation*}
Using the summation formulas \eqref{eq:sumubwE} and \eqref{approx-A-Hessian}, we can write the above quantity as
\begin{equation*}
\begin{aligned}
C^P_2 \int_\Omega &2\K_P^2 \gamma(\bx) \Im m\{\nabla \bGam_{0,P}(\bz,\bx)\}: [\MM'(B') \Im m\{\nabla_\bz \bGam_{0,P}(\bz,\bx)\}]\\
&\quad \quad \quad +2\Im m\{ \nabla\nabla \bGam_{0,P}(\bz,\bx)\} : [\MM'(B') \Im m\{\nabla_\bz\bGam_{0,P}(\bz,\bx) \otimes \nabla \gamma(\bx)\}] d\bx.
\end{aligned}
\end{equation*}
Here, the constant has the expression
\begin{equation*}
C^P_2 = \frac{1}{\omega} 4\mu_0 (\frac{\pi}{\K_P})^{d-2} (\frac{\K_S}{\K_P})^2 = 4\pi^{d-2} \frac{\rho\omega}{\K_P^d}.
\end{equation*}
The tensor products above are defined as
\begin{equation*}
\begin{aligned}
(\MM'(B') \Im m\{\nabla_\bz \bGam_{0,P}(\bz,\bx)\})_{jkl} &= m_{jkpq} \partial_p (\Im m\{\bGam_{0,P}(\bz,\bx)\})_{ql},\\
(\MM'(B') \Im m\{\nabla_\bz\bGam_{0,P}(\bz,\bx) \otimes \nabla \gamma(\bx))_{jkls} &= (\MM'(B') \Im m\{\nabla_\bz \bGam_{0,P}(\bz,\bx)\})_{jkl} \partial_s \gamma,\\
(\nabla\nabla \bGam_{0,P}(\bz,\bx))_{jkls} &= \partial_j \partial_k (\bGam_{0,P}(\bz,\bx))_{ls}.
\end{aligned}
\end{equation*}

Similarly, when a set of shear plane waves $\{\bU^S_j\}$ are used, the speckle field in the imaging function can be calculated as follows.
\begin{equation*}
\begin{aligned}
&\I_{\mathrm{WF,noise}}[\{\bU^S_j\}](\bz) =  c_S\frac{1}{n} \Re e \sum_{j=1}^n \nabla \bU_j^S(\bz) : \MM'(B') \nabla \Bw^S_{\mathrm{noise}}(\bz)\\
=\ &\frac{1}{\omega}\Re e \int_\Omega -\gamma(\bx) \frac{\K_S^2 }{n} \sum_{j=1}^n i\K_S e^{i\K_S (\bz - \bx)\cdot \be_{\theta_j}} \be_{\theta_j} \otimes \be_{\theta_j}^\perp \otimes \be_{\theta_j}^\perp: \MM'(B')\Im m \{\nabla_\bz \bGam_{0,S}(\bz,\bx)\}\\
&-\frac{1}{n} \sum_{j=1}^n (i\K_S)^2 e^{i\K_S(\bz-\bx)\cdot \be_{\theta_j}} \be_{\theta_j} \otimes \be_{\theta_j} \otimes \be_{\theta_j}^\perp \otimes \be_{\theta_j}^\perp : [\nabla \gamma(\bx)\otimes  \MM'(B')\Im m \{\nabla_\bz \bGam_{0,S}(\bz,\bx)\}]\\
&-\frac{1}{n} \sum_{j=1}^n (i\K_S)^2 e^{i\K_S(\bz-\bx)\cdot \be_{\theta_j}} \be_{\theta_j} \otimes \be_{\theta_j}^\perp \otimes \be_{\theta_j} \otimes \be_{\theta_j}^\perp : [\MM'(B')\Im m \{\nabla_\bz \bGam_{0,P}(\bz,\bx)\} \otimes \nabla \gamma(\bx)] d\bx.
\end{aligned}
\end{equation*}
Use formula \eqref{eq:sumubwE} for the first integral and the differentiation of \eqref{eq:sumubwE} for the second integral; use formula \eqref{approx-eP-times-eP-Hessian} for the third integral. We conclude that the speckle filed admits the following expression:
\begin{equation*}
\begin{aligned}
C^S_2 \int_\Omega &\K_S^2 \gamma(\bx) \Im m\{\nabla \bGam_{0,S}(\bz,\bx)\}: [\MM'(B') \Im m\{\nabla_\bz \bGam_{0,S}(\bz,\bx)\}]\\
&\quad \quad \quad +\Im m\{ \nabla\nabla \bGam_{0,S}(\bz,\bx)\} : [\nabla \gamma(\bx)  \otimes \MM'(B') \Im m\{\nabla_\bz\bGam_{0,P}(\bz,\bx)\}]\\
&\quad \quad \quad +\Im m\{ \nabla^2 \bGam_{0,S}(\bz,\bx)\} : [\nabla \gamma(\bx)  \otimes \MM'(B') \Im m\{\nabla_\bz\bGam_{0,P}(\bz,\bx)\}] d\bx.
\end{aligned}
\end{equation*}
Here, the constant has the expression
\begin{equation*}
C^S_2 = \frac{1}{\omega} 4\mu_0 (\frac{\pi}{\K_S})^{d-2} = 4\pi^{d-2} \frac{\rho\omega}{\K_S^d}.
\end{equation*}
Note also that the tensor $\nabla\nabla \bGam_{0,S}$ has coordinates $\{\partial_i \partial_j (\bGam_{0,S})_{kl}\}$ and it is different from $\nabla^2 \bGam_{0,S}$ which is defined first in \eqref{approx-eP-times-eP-Hessian}.

\begin{rem} To summarize, we derived expressions for the speckle field in the imaging functional when the elastic medium has random fluctuations. We showed that the speckle field is essentially the medium noise (and its derivatives) smoothed by some integral kernels that depend only on the fundamental solution in the homogeneous medium. As a consequence, the covariance function of the speckle filed can be expressed as that of the medium noise (and those of its derivatives) smoothed by similar integral kernels.
\end{rem}
\section{Conclusion}\label{sec:conclu}

In this paper, we performed an analysis of the topological
derivative (TD) based elastic inclusion detection algorithm. We
have seen that the standard TD based imaging functional may not
attain its maximum at the location of the inclusion. Moreover, we
have shown that its resolution does not reach the diffraction
limit and identified the responsible terms, that are associated to
the coupling of different wave-modes. In order to enhance
resolution to its optimum, we cancelled out these coupling terms
by means of a Helmholtz decomposition and thereby designing a
weighted imaging functional. We proved that the modified
functional behaves like the square of the imaginary part of a
pressure or a shear Green function, depending upon the choice of
the incident wave, and then attains its maximum at the true
location of the inclusion with a Rayleigh resolution limit, that
is, of the order of half a wavelength. Finally, we have shown that
the proposed imaging functionals are very stable with respect to
measurement noise and moderately stable with respect to medium
noise.

In a forthcoming work, we intend  to extend the results of the
paper to the localization of the small infinitesimal elastic
cracks and to the case of elastostatics. In this regard recent
contributions \cite{crack, AGJK-Top, AK-Pol} are expected to play
a key role.

\end{document}